\documentclass{article}
\usepackage[utf8]{inputenc}
\usepackage{pig}
\usepackage{fullpage}
\usepackage{enumitem}
\usepackage{placeins}
\usepackage{subcaption}
\usepackage{graphicx}
\usepackage{caption}

\title{Statistical limits of spiked tensor models}

\usepackage{authblk}
\author[1]{Amelia Perry%
\footnote{The first two authors contributed equally.}%
\thanks{Email: {\tt ameliaperry@mit.edu}. This work is supported in part by NSF CAREER Award CCF-1453261 and a grant from the MIT NEC Corporation.}%
}
\newcommand{\firstauthmark}{\footnotemark[1]}
\author[1]{Alexander S.\ Wein%
\protect\firstauthmark%
\thanks{Email: {\tt awein@mit.edu}. This research was conducted with Government support under and awarded by DoD, Air Force Office of Scientific Research, National Defense Science and Engineering Graduate (NDSEG) Fellowship, 32 CFR 168a.}%
}
\author[2]{Afonso S.\ Bandeira%
\thanks{Email: {\tt bandeira@cims.nyu.edu}.
}}
\affil[1]{Department of Mathematics, Massachusetts Institute of Technology}
\affil[2]{Department of Mathematics and Center for Data Science, Courant Institute of Mathematical Sciences, New York University}

\begin{document}
\maketitle

\begin{abstract}

We study the statistical limits of both detecting and estimating a rank-one deformation of a symmetric random Gaussian tensor. We establish upper and lower bounds on the critical signal-to-noise ratio, under a variety of priors for the planted vector: (i) a uniformly sampled unit vector, (ii) \iid $\pm 1$ entries, and (iii) a sparse vector where a constant fraction $\rho$ of entries are \iid $\pm 1$ and the rest are zero. For each of these cases, our upper and lower bounds match up to a $1+o(1)$ factor as the order $d$ of the tensor becomes large. For sparse signals (iii), our bounds are also asymptotically tight in the sparse limit $\rho \to 0$ for any fixed $d$ (including the $d=2$ case of sparse PCA). Our upper bounds for (i) demonstrate a phenomenon reminiscent of the work of Baik, Ben Arous and P\'ech\'e: an `eigenvalue' of a perturbed tensor emerges from the bulk at a strictly lower signal-to-noise ratio than when the perturbation itself exceeds the bulk; we quantify the size of this effect. We also provide some general results for larger classes of priors. In particular, the large $d$ asymptotics of the threshold location differs between problems with discrete priors versus continuous priors. Finally, for priors (i) and (ii) we carry out the replica prediction from statistical physics, which is conjectured to give the exact information-theoretic threshold for any fixed $d$.

Of independent interest, we introduce a new improvement to the second moment method for contiguity, on which our lower bounds are based. Our technique conditions away from rare `bad' events that depend on interactions between the signal and noise. This enables us to close $\sqrt{2}$-factor gaps present in several previous works.

\end{abstract}

\vspace{0.1em}

\section{Introduction}

Among the central problems in random matrix theory is to determine the spectral properties of \emph{spiked} or \emph{deformed} random matrix ensembles. Introduced by Johnstone \cite{johnstone}, such matrices consist of a random matrix (e.g.\ Wigner or Wishart) with a low-rank perturbation. These distributions serve as models for data consisting of ``signal plus noise'' and thus, results on these models form the basis of our theoretical understanding of principal components analysis (PCA) throughout the sciences.

Perhaps the most studied phenomenon in these spiked ensembles is the transition first examined by \cite{bbp} in the Wishart setting. We will be interested in the Wigner analogue: if $W$ is a Gaussian Wigner matrix\footnote{$W$ is symmetric with off-diagonal entries $\cN(0,1/n)$, diagonal entries $\cN(0,2/n)$, and all entries independent (except for symmetry).} and $x$ is a unit vector (the `spike'), the spectrum of the spiked matrix $\lambda xx^\top + W$ undergoes a sharp phase transition at $\lambda = 1$ (see e.g.\ \cite{peche,fp,cdf,eig-vec}). Namely, when $\lambda \le 1$, many properties of the spectrum resemble those of a random (`unspiked') matrix: the empirical distribution of eigenvalues is semicircular and the top eigenvalue concentrates about $2$. When $\lambda > 1$, however, the spectrum becomes indicative of the spike: a single eigenvalue exceeds 2, exiting the semicircular bulk, and the associated eigenvector is correlated with the spike (and the precise correlation is known).

We emphasize that this `BBP-style' transition \cite{bbp} exhibits a \emph{push-out} effect: the top eigenvalue of a random Wigner matrix is $2$ (in the high-dimensional limit), but one only needs to add a planted signal of spectral norm $1$ before it becomes visible in the spectrum \cite{fp}. Once $\lambda > 1$, the planted signal aligns well enough with fluctuations of the random matrix in order to create an eigenvalue greater than $2$.

More recent work shows a second, statistical role of this $\lambda=1$ threshold: not only does the top eigenvalue fail to distinguish the spiked and unspiked models for $\lambda < 1$, but in fact every hypothesis test fails with constant probability \cite{omh,mrz,pwbm}. Thus, this transition indicates the point at which the spiked and unspiked models become markedly different.

It is natural to ask how the phenomena above generalize to tensors of higher order. Such tensors lack a well-behaved spectral theory, and many standard tools of random matrix theory (e.g.\ the method of resolvents) fail to cleanly generalize. However, there remain a number of interesting probabilistic questions to ask in this setting:
\begin{itemize}
\item Is there a sharp transition point for $\lambda$, below which the spiked model resembles the unspiked model?

In particular, we compare the Wigner tensor model (entries of $W$ are Gaussian, \iid apart from permutation symmetry; see Definition~\ref{def:wig-tensor}) to a spiked analogue $T = \lambda x^{\otimes d} + W$. Previous work of \cite{mrz} provides a bound on $\lambda$ below which these two models are information-theoretically indistinguishable. On the other hand, \cite{rm} notes that once $\lambda$ exceeds the injective norm (defined below) of the noise, the spiked and unspiked models can be distinguished via injective norm. There remained a $\sqrt{2}$-factor gap between these bounds as the order $d \to \infty$. In this paper, we improve both the lower and upper bounds, saving an asymptotic $\sqrt{2}$ factor in the lower bound, and thus obtaining a $(1 + o(1))$ factor gap as $d \to \infty$; see Figure~\ref{fig:compare}.

\item By analogy with the top eigenvalue (spectral norm) of a random matrix, what is the \emph{injective norm}
$$ \|T\| \defeq \max_{\|x\|=1} \langle T, x^{\otimes d} \rangle $$
of a spiked Wigner $d$-tensor?

For unspiked tensors, the value of the injective norm was predicted by \cite{cs92} through non-rigorous methods from statistical physics, and was later rigorously proven \cite{tal06,abac,subag-pspin}. However, the spiked question has not been studied to our knowledge. Our statistical lower bound shows in particular that the injective norm of a spiked tensor remains identical to this value for all $\lambda$ below a critical value. However, starting slightly above this critical value, our work provides a strong lower bound on the injective norm of the spiked model $T = \lambda x^{\otimes d} + W$, which exceeds the injective norm of the unspiked model.

\item Do tensors exhibit a BBP-style push-out effect, and if so, how large is it?

We show that the injective norm of a spiked tensor exceeds that of an unspiked tensor strictly before the injective norm of the spike exceeds that of the noise, much as the $\lambda=1$ threshold in the matrix case strictly precedes the spectral norm of $2$ in the noise. We identify the asymptotic size of this gap as the order $d$ becomes large, up to a small constant factor.

\end{itemize}

\begin{figure}[!ht]
    \centering
    \begin{subfigure}[t]{0.47\textwidth}
        \centering
        \includegraphics[width=\linewidth]{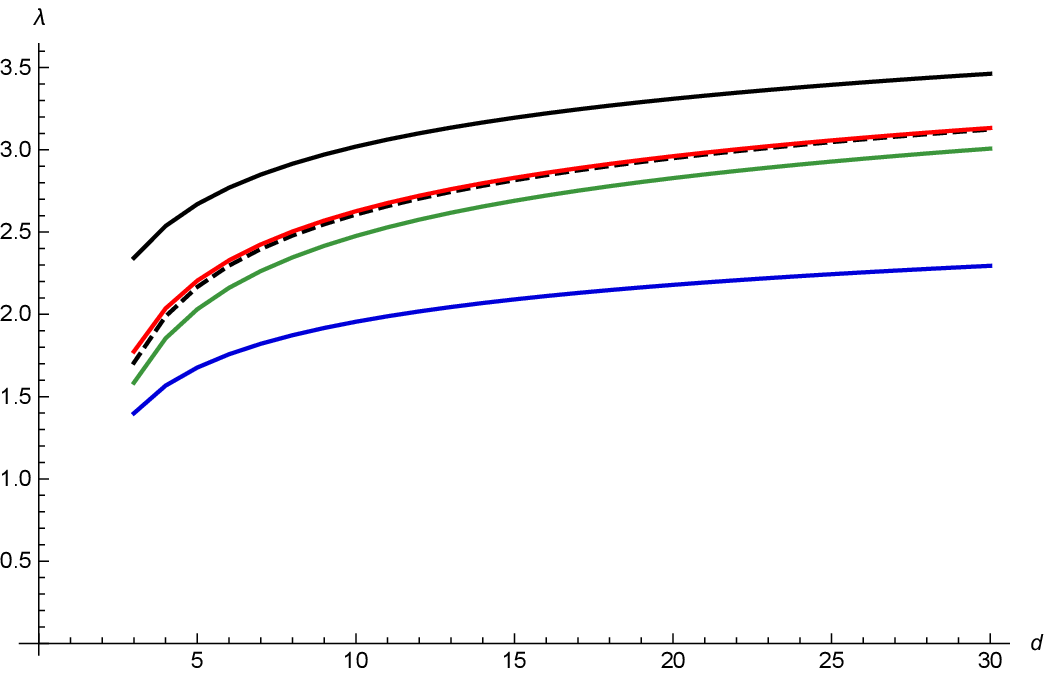}
    \end{subfigure}
    \hfill
    \begin{subfigure}[t]{0.47\textwidth}
        \centering
        \includegraphics[width=\linewidth]{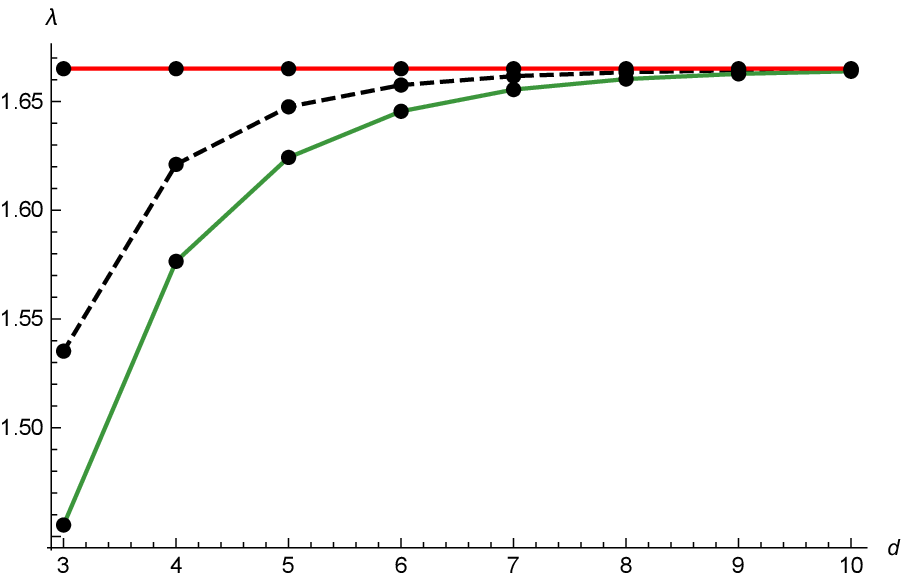}
    \end{subfigure}
    \caption{Critical SNR values $\lambda$ as a function of the tensor order $d$. Left: the spherical prior; from top to bottom, black: the injective norm $\mu_d$ of a noise tensor (see Theorem~\ref{thm:inj-norm}); red: our upper bound $\Lambda^*_{\sph,d}$ (Corollary~\ref{cor:sph-upper}); dashed: replica prediction for the exact threshold (Appendix~\ref{app:replica}); green: our lower bound $\lambda^*_{\sph,d}$ (Theorem~\ref{thm:sph-lower}); blue: the lower bound of \cite{mrz} which is loose by a factor of $\sqrt{2}$ for large $d$. Right: the Rademacher prior; from top to bottom, red: our upper bound (Theorem~\ref{thm:dinf-rademacher}); dashed: replica prediction, which is a rigorous upper bound (see Appendix~\ref{app:replica}); green: our lower bound $\lambda^*_{\Rade,d}$ (Theorem~\ref{thm:rad-lower}).}
    \label{fig:compare}
\end{figure}

\begin{figure}[!ht]
    \centering
    \includegraphics[width=0.7\linewidth]{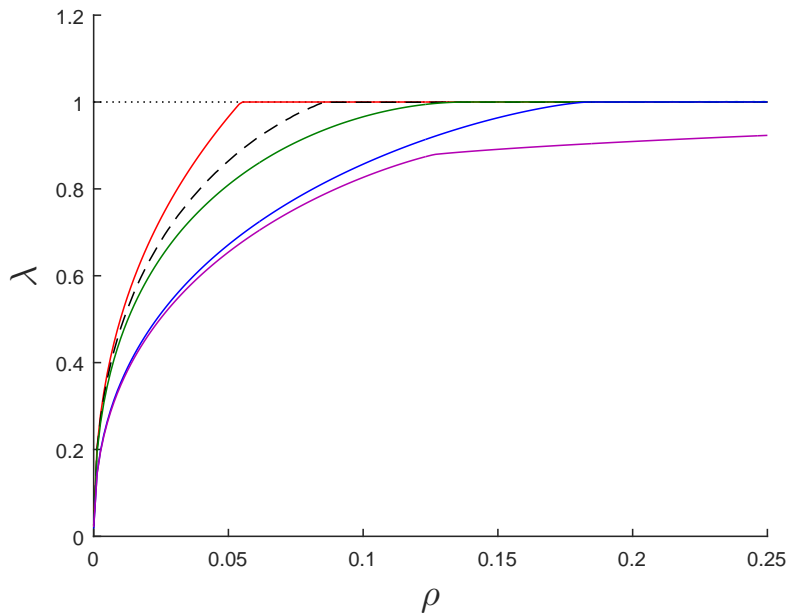}
    \caption{Upper and lower bounds for the critical $\lambda$ value in sparse Rademacher PCA ($d=2$) as a function of sparsity $\rho$. Note that $\lambda = 1$ corresponds to the eigenvalue transition above which efficient algorithms (such as PCA) are known to succeed. From top to bottom, red: MLE upper bound (Theorem~\ref{thm:dinf-sparse-rademacher}, first due to \cite{bmvvx}); dashed: replica prediction of \cite{mmse-lowrank} which has been rigorously shown to be correct for the weak recovery problem \cite{mi,mi-proof,lelarge-limits-lowrank}; green: our noise conditioning lower bound $\lambda^*_{\spRade(\rho),d}$ (Theorem~\ref{thm:sparse-rad-lower}); blue: our lower bound without noise conditioning \cite{pwbm}; purple: lower bound of \cite{bmvx} (equations (33),(34) in \cite{bmvx}). In the limit $\rho \to 0$, noise conditioning (green) achieves the correct asymptotic behavior (matching the upper bound) whereas the blue and purple curves are loose by a factor of $\sqrt{2}$.}
    \label{fig:sparse-rad}
\end{figure}

Much as random matrix theory provides a theoretical foundation for PCA, these questions probe at the statistical limits of \emph{tensor PCA}, the estimation of a low-rank spike from a spiked tensor. Such problems arise naturally in topic modeling \cite{anandkumar2014tensor}, in the study of hypergraphs \cite{duchenne2011tensor}, and more generally in moment-based approaches to estimation, when it may be desirable to detect low-rank structure in higher empirical moments. As many such problems involve extra structure such as sparsity in the signal, we allow the spike to be drawn from various priors, such as a distribution of sparse vectors. For each prior, we will investigate the detection problem of distinguishing the spiked and unspiked models, as well as the recovery problem of estimating the spike.

\paragraph{Our techniques and prior work.}

We will prove information-theoretic lower bounds using the second moment method associated with the statistical notion of \emph{contiguity}. By computing a particular second moment, one can show that the spiked and unspiked models are contiguous, implying that they cannot be reliably distinguished. This second moment method originated in the study of random graphs (see e.g.\ \cite{wor-survey,rw-ham,janson,1-fact}) but has since been applied to various average-case computational problems such as the stochastic block model \cite{mns,bmnn}, submatrix localization \cite{bmvvx}, Gaussian mixture clustering \cite{bmvvx}, synchronization problems over compact groups \cite{pwbm}, and even spiked tensor PCA \cite{mrz} and sparse PCA \cite{bmvvx}. However, many of the previous results are not tight in particular asymptotic regimes. In fact, there are many instances where curiously, in certain regimes they are loose by precisely a factor of $\sqrt{2}$ in the signal-to-noise ratio $\lambda$ \cite{mrz,bmnn,bmvx,pwbm}.

Our main technical contribution is a modification of the second moment method that closes (at least some of) these gaps. Specifically, we close the gap for spherically-spiked\footnote{Here, the spike is a uniformly random unit vector.} tensor PCA in the limit of large tensor order \cite{mrz}, and the gap\footnote{A more recent update \cite{bmvvx} independently closes the asymptotic $\sqrt{2}$-factor gap from \cite{bmvx}, using a different modification of the second moment method.} for sparse PCA in the limit of low sparsity \cite{bmvx}. Our technique, which we call \emph{noise conditioning}, is based on conditioning away from rare bad events that depend jointly on the signal and noise. We expect that this technique can also be used to close several other $\sqrt{2}$-factor gaps of the same nature. Another application in which we have found noise conditioning to be fruitful is in contiguity results for the Rademacher-spiked Wishart model \cite{pwbm}. For this probem the basic second moment method struggles quite badly because, due to a certain symmetry, it gives the same results for the positively- and negatively-spiked regimes, even though these two regimes have very different thresholds; the noise conditioning method is able to break this symmetry, yielding tight or almost-tight results in both regimes.

Our noise conditioning method is somewhat reminiscent of other modified second moment methods that have appeared in other contexts such as branching Brownian motion \cite{bramson-branching}, branching random walks \cite{aidekon-branching,conv-branching}, the Gaussian free field \cite{bdg-gff,tightness-gaussian-field,conv-gaussian-field}, cover times for random walks \cite{cover-times}, and thresholds for random satisfiability problems (e.g.\ $k$-colorability, $k$-sat) \cite{condensation-hypergraph,catching-naesat,chasing-k-colorability,going-after-ksat,asymptotic-ksat}.

Our upper bound for the spherical prior, via a lower bound on the spiked injective norm, is based on a direct analysis of how vectors close to the spike can align constructively with fluctuations in the noise to produce a larger injective norm than the spike alone. In particular, we consider how such vectors align with submatrices of the given tensor, and leverage existing results from spiked matrix models. Upper bounds for structured priors are obtained through na\"ive union bounds on the maximum likelihood value.

A variety of other techniques originating from statistical physics have also been successful in tackling structured inference or optimization problems involving large random systems, including random matrices and tensors. For instance, random tensors are intimately connected to spin glasses with $d$-spin interactions \cite{gardner,cs92}. The so-called \emph{replica method} gives extremely precise non-rigorous solutions to these types of problems (see \cite{mm-book} for an introduction). In some cases, such as the celebrated Parisi formula for the ground state of a Sherrington--Kirkpatrick spin glass\footnote{This can be thought of as the maximum value of $x^\top W x$ over $x \in \{\pm 1\}^n$ where $W$ is a Gaussian Wigner matrix.}, the replica prediction has been rigorously proven to be correct \cite{tal06parisi}; we also rigorously know the injective norm of a random tensor of any order (in the high-dimensional limit), as well as various structural properties of the critical points of the associated maximization problem \cite{tal06,abac,subag-pspin}. Furthermore, for a variety of structured spiked matrix problems such as sparse PCA (with constant-fraction sparsity), the statistically-optimal mean squared error can be exactly characterized in the high-dimensional limit for any level of sparsity and any signal-to-noise ratio \cite{mi,mi-proof,lelarge-limits-lowrank}. Additionally, \cite{km} implies bounds on the statistical threshold for tensor PCA with a $\{\pm 1\}$-valued spike (which we discuss in Appendix~\ref{app:replica}). Some key techniques used in the above works include Guerra interpolation \cite{guerra}, the Kac--Rice formula (see \cite{abac}), and the approximate message passing (AMP) framework \cite{amp-cs,amp-mot,bm,jm} (see also \cite{sparse-pca-amp,dam,nonneg-pca,mi-proof}).

In comparison to these techniques from statistical physics, the second moment method for contiguity typically does not yield results that are as sharp, but it has the advantage of being quite simple and widely applicable. In particular, it can be applied to problems such as the sparse stochastic block model \cite{mns,bmnn} (constant average degree) for which various techniques in statistical physics do not seem to apply (see e.g.\ \cite{dam,lelarge-limits-lowrank}). Another advantage of the second moment method is that it addresses the detection problem instead of only the recovery problem, and furthermore implies bounds on hypothesis testing power below the detection threshold \cite{pwbm}. (However, unlike e.g.\ AMP, the second moment method only tells us about the threshold for nontrivial recovery and not the optimal recovery error at each value of $\lambda$ above the threshold.)

We remark that our results on statistical indistinguishability have concrete implications for various probabilistic quantities. In particular, for any quantity that converges in probability to a constant under the spiked and unspiked models (such as the injective norm), the limiting values must agree throughout the subcritical region of the spiked model (signal strength below the detection threshold)\footnote{In fact, for tensors of order $d \ge 3$ we show that the unspiked and subcritical-spiked distributions differ by $o(1)$ in total variation distance (in the high-dimensional limit), implying that any quantity with a limit in distribution must converge to the same distribution under both models.}. For example, in the unspiked model, \cite{abac,subag-pspin} give a detailed description of the energy landscape, i.e.\ the number of critical points of $x \mapsto \langle T, x^{\otimes d} \rangle$ of any given value and index; our results immediately imply that the same is true for subcritical spiked tensors. An interesting problem, that we largely do not address here, is to characterize the energy landscape above the detection threshold.

It is important to note that we are studying information-theoretic limits rather than computational ones. All of the upper bounds in this paper are inefficient algorithms such as exhaustive search over all possible spikes. There is good evidence in the form of sum-of-squares lower bounds that there is a significant gap between what is possible statistically and what is possible computationally for tensor PCA and related problems \cite{sos-tensor,sos-clique}. The $d=2$ case of matrices is of course an exception; in this case spectral algorithms achieve the optimal detection and recovery threshold in many cases. Various efficient algorithms for spiked tensor PCA are considered by \cite{rm}, but (as is believed to be necessary) these operate only in a regime that is quite far from the information-theoretic threshold.

\subsection{Preliminaries}

\begin{definition}
\label{def:wig-tensor}
We define a \emph{Wigner tensor} $W \in \RR^{n^d}$ of order $d$ by the following sampling procedure: an asymmetric precursor $W' \in \RR^{n^d}$ is drawn with entries sampled \iid from $\cN(0,2/n)$, and is then averaged over all permutations of the indices to form a symmetric tensor $W$. Thus, the distribution of $W$ has density proportional to $\exp(-\frac{n}{4} \langle W,W \rangle)$ on the space of symmetric tensors. We denote this distribution by $\WT(d)$.
\end{definition}

\noindent Note that a typical entry of $W$ (with no repeated indices) is distributed as $\cN(0,2/nd!)$, and for any unit vector $x$ we have $\langle x^{\otimes d}, W \rangle \sim \cN(0,2/n)$. This normalization agrees with that of \cite{mrz}, but other conventions exist in the literature.

Let the \emph{prior} $\cX = \{\cX_n\}$ be a family of distributions (one for each $n$) over unit vectors $x \in \RR^n$, $\|x\|=1$. Define the \emph{spiked} distribution $\WT(d,\lambda,\cX)$ by $T = \lambda x^{\otimes d} + W$, where $x$ is drawn from $\cX$, $\lambda \ge 0$ is a signal-to-noise parameter, and $W$ is drawn from $\WT(d)$.

The injective norm of a $d$-tensor $T$ is defined as $\|T\| = \max \langle x^{\otimes d}, T \rangle$ over unit vectors $x$. For each $d \ge 2$ it is known that the injective norm $\|\WT(d)\|$ of a random tensor converges in probability to a particular value $\mu_d$ (see Theorem~\ref{thm:inj-norm}) \cite{cs92,tal06,abac,rm,subag-pspin}. For $d=2$ we have $\mu_2 = 2$, the spectral norm of a Wigner matrix. We also have, for instance, $\mu_3 \approx 2.3433$, and $\mu_d = (1+o(1)) \sqrt{2 \log d}$ as $d\to\infty$.

In the matrix case $(d=2)$ we have the following classical eigenvalue transition for spiked Wigner matrices.
\begin{theorem}[\cite{peche,fp,cdf,eig-vec}]
\label{thm:bbp-wig}
Let $T = \lambda x x^\top + W$ where $x$ is a unit vector in $\RR^n$ and $W$ is an $n \times n$ Gaussian Wigner matrix (symmetric with off-diagonal entries $\cN(0,1/n)$, diagonal entries $\cN(0,2/n)$, and all entries independent up to symmetry).
\begin{itemize}
\item If $\lambda \le 1$, the top eigenvalue of $T$ converges almost surely to $2$ as $n \to \infty$, and the top eigenvector $v$ (unit-norm) has trivial correlation with the spike: $\langle v,x \rangle^2 \to 0$ almost surely.
\item If $\lambda > 1$, the top eigenvalue converges almost surely to $\lambda + 1/\lambda > 2$ and $v$ has nontrivial correlation with the spike: $\langle v,x \rangle^2 \to 1 - 1/\lambda^2$ almost surely.
\end{itemize}
\end{theorem}

All of our results will consider the $n \to \infty$ limit. We will be interested in the following detection and recovery problems. For the detection problems, we are given a single sample from either the spiked distribution $\WT(d,\lambda,\cX)$ or the corresponding unspiked distribution $\WT(d)$ (say each is chosen with probability $1/2$) and our goal is to decide which distribution the sample came from.
\begin{itemize}
\item{\bf Strong detection}: Distinguish between the spiked and unspiked distributions with success probability $1-o(1)$ as $n \to \infty$.
\item{\bf Weak detection}: Distinguish between the spiked and unspiked distributions with success probability $\frac{1}{2}+\eps$ as $n \to \infty$, for some $\eps > 0$ that does not depend on $n$.
\item{\bf Weak recovery}: Given a sample from the spiked model, output a unit vector $\hat x$ with nontrivial correlation with the spike: $\EE[\langle x,\hat x \rangle^d] \ge \eps$ for some $\eps > 0$ that does not depend on $n$.
\end{itemize}

\noindent Note that the exponent of $d$ in $\langle x,\hat x \rangle^d$ captures the fact that when $d$ is even we can only hope to learn the spike up to a global sign flip.

The three problems above are related as follows. Clearly strong detection implies weak detection, but there is no formal connection between detection and recovery in general (see e.g.\ \cite{bmvvx} for a simple counterexample). Typically, however, strong detection and weak recovery tend to be equivalent in the sense that they exhibit the same threshold for $\lambda$ above which they are possible. For spiked tensors with $d \ge 3$, we will see that this also tends to coincide with the weak detection threshold. However, for matrices ($d = 2$), weak detection is actually possible below the strong detection threshold; in fact it is possible for any $\lambda > 0$ simply by inspecting the trace of the matrix.

\subsection{Summary of results}
\label{sec:summary}

Our results apply to a wide range of spike priors, but here we present specific results for three priors:
\begin{enumerate}[label=(\roman*)]
\item the spherical prior $\cXs$, in which $x$ is drawn uniformly from the unit sphere in $\RR^n$,
\item the Rademacher prior $\cX_\Rade$, in which the entries of $x$ are drawn \iid from $\{\pm 1/\sqrt{n} \}$,
\item the sparse Rademacher prior $\cX_\spRade(\rho)$ for $\rho \in (0,1]$, in which a random support of exactly $\rho n$ entries is chosen uniformly, and those entries of $x$ are drawn \iid from $\{\pm 1/\sqrt{\rho n} \}$ (with all other entries $0$).
\end{enumerate}

\noindent Although we will give bounds for every $d$, our results are most precisely stated in the limit $d \to \infty$ in which case the lower and upper bounds match asymptotically.

\begin{theorem}[Spherical prior, large $d$]\label{thm:dinf-spherical} Consider the spherical prior $\cX_\sph$. There exist bounds $\lambda^*_{\sph,d}$ and $\Lambda^*_{\sph,d}$, both behaving as $\sqrt{2 \log d} + o(1)$ as $d \to \infty$, and with $\Lambda^*_{\sph,d} < \mu_d$ for each $d$, such that
\begin{itemize}
    \item if $\lambda < \lambda^*_{\mathrm{sph},d}$ then
    weak detection and weak recovery are impossible,
    \item if $\lambda > \Lambda^*_{\mathrm{sph},d}$ then
    strong detection and weak recovery are possible.
\end{itemize}
\end{theorem}

\noindent Recall that $\mu_d$ is the value of the injective norm of $\WT(d)$, which also behaves as $\sqrt{2\log d} + o(1)$; see Theorem~\ref{thm:inj-norm} for a formal statement, including the value of $\mu_d$. Explicit formulas for the bounds are given by Theorem~\ref{thm:sph-lower} ($\lambda^*_{\sph,d}$) and Theorem~\ref{thm:sph-upper} ($\Lambda^*_{\sph,d}$). We obtain more detailed asymptotic descriptions in Appendix~\ref{app:asymptotics}:
\begin{align*}
    \mu_d^2 &= 2 \log d + 2 \log \log d + 2 + o(1), \\
    (\Lambda_{\sph,d}^*)^2 &= 2 \log d + 2 \log \log d + o(1), \\
    (\lambda_{\sph,d}^*)^2 &= 2 \log d + 2 \log \log d + 2 - 4 \log 2 + o(1)
\end{align*}
in the limit $d \to \infty$. Hence the quantities $\mu_d$, $\Lambda_{\sph,d}$, and $\lambda_{\sph,d}$ all behave as $\sqrt{2 \log d} + o(1)$. Our lower bound closes a $\sqrt{2}$-factor gap in \cite{mrz}. A non-rigorous calculation via the replica method (Appendix~\ref{app:replica}) suggests that the true statistical threshold matches the asymptotics $2 \log d + 2 \log \log d + o(1)$ of the upper bound. These various quantities are depicted in Figure~\ref{fig:compare}.

It is clear that once $\lambda > \mu_d$, strong detection is possible (by thresholding the injective norm). Our upper bound implies that for each $d$, the spherically-spiked and unspiked tensor models can be distinguished as soon as $\lambda > \Lambda^*_{\sph,d}$, a threshold strictly below $\mu_d$; indeed, we show (Theorem~\ref{thm:sph-upper}) that the injective norm of $\WT(d,\lambda,\cXs)$ exceeds that of $\WT(d)$ for such $\lambda$. This mirrors the eigenvalue transition for $d=2$ (see Theorem~\ref{thm:bbp-wig}), in which an eigenvalue leaves the spectrum bulk when $\lambda > 1$, exhibiting a gap from when the spike exceeds the bulk in spectral norm at $\lambda = 2$. The results above imply that the size of this ``BBP gap'' lies between $\mu_d - \Lambda^*_{\sph,d}$ and $\mu_d - \lambda^*_{\sph,d}$ and is therefore of order $1/\sqrt{\log d}$ as $d \to \infty$.

We now present our results for discrete priors. One qualitative difference from the above is that for discrete priors, the statistical threshold tends to remain bounded as $d \to \infty$:

\begin{theorem}[Rademacher prior, large $d$]\label{thm:dinf-rademacher}
Consider the Rademacher prior $\cX_\Rade$. There exists a bound $\lambda^*_{\Rade,d}$, with $\lim_{d \to \infty} \lambda^*_{\Rade,d} = 2 \sqrt{\log 2}$, such that
\begin{itemize}
    \item if $\lambda < \lambda^*_{\mathrm{\Rade},d}$ then
    weak detection and weak recovery are impossible,
    \item if $\lambda > 2\sqrt{\log 2}$ then
    strong detection and weak recovery are possible.
\end{itemize}
\end{theorem}

\begin{theorem}[Sparse Rademacher prior, large $d$]\label{thm:dinf-sparse-rademacher}
Consider the sparse Rademacher prior $\cX_\spRade(\rho)$ with $\rho \in (0,1]$ fixed. There exists a bound $\lambda^*_{\spRade(\rho),d}$, with $\lim_{d \to \infty} \lambda^*_{\spRade(\rho),d} = 2 \sqrt{H(\rho) + \rho \log 2}$, such that
\begin{itemize}
    \item if $\lambda < \lambda^*_{\spRade(\rho),d}$ then
    weak detection and weak recovery are impossible,
    \item if $\lambda > 2\sqrt{H(\rho) + \rho \log 2}$ then
    strong detection and weak recovery are possible.
\end{itemize}
Here $H$ denotes the binary entropy: $H(\rho) = -\rho \log \rho - (1-\rho)\log(1-\rho)$.
\end{theorem}

\noindent See Theorems~\ref{thm:rad-lower} and \ref{thm:sparse-rad-lower} for explicit formulas for the lower bounds. The upper bounds follow from Proposition~\ref{prop:cardinality-upper-bound}. We remark that the hardest case for sparse Rademacher is $\rho = 2/3$ (maximizing $H(\rho) + \rho \log 2$), where the three types of entries ($1/\sqrt{\rho n}, -1/\sqrt{\rho n}, 0$) occur in equal proportion. In the above three cases (spherical, Rademacher, sparse Rademacher), the lower and upper bounds agree up to a $1+o(1)$ factor as $d \to \infty$. The bounds $\lambda^*,\Lambda^*$ for each $d$ are described as finite-dimensional optimization problems and are easy to compute numerically. See Figure~\ref{fig:compare} for a comparison of the bounds for various values of $d$.

The special case of $d=2$ has been previously understood for the spherical and Rademacher priors \cite{dam,mrz,pwbm}. Namely, the threshold for both strong detection and weak recovery occurs precisely at $\lambda = 1$, matching the eigenvalue transition.

Another asymptotic regime that we consider is the sparse Rademacher prior with $d$ fixed and sparsity $\rho \to 0$, i.e.\ the limit of extremely sparse vectors.

\begin{theorem}[Sparse Rademacher prior, $\rho \to 0$]\label{thm:p0-sparse-rademacher}
Fix $d \ge 2$ and consider $\cX_\spRade(\rho)$. There exists a bound $\lambda^*_{\spRade(\rho),d}$ that behaves as $2\sqrt{-\rho \log \rho + O(\rho)}$ in the limit $\rho \to 0$, such that
\begin{itemize}
    \item if $\lambda < \lambda^*_{\spRade(\rho),d}$ then
    strong detection is impossible,
    \item if $\lambda > 2\sqrt{H(\rho) + \rho \log 2} = 2\sqrt{-\rho \log \rho + O(\rho)}$ then
    strong detection and weak recovery are possible.
\end{itemize}
For $d \ge 3$, the lower bound also rules out weak detection and weak recovery.
\end{theorem}

\noindent The $d = 2$ case of the above theorem was previously considered by \cite{bmvx} where they give a tight upper bound and a lower bound that is loose\footnote{A more recent update \cite{bmvvx} independently closes the asymptotic $\sqrt{2}$-factor gap from \cite{bmvx}, using a different modification of the second moment method.} by a factor of $\sqrt{2}$. Here we close the $\sqrt{2}$ gap by improving the lower bound. Our upper bounds for discrete priors are straightforward generalizations of their upper bound, based on exhaustive search over all possible spikes. The recovery problem for the sparse Rademacher prior (with $d=2$) has also been studied using tools from statistical physics. In particular, the weak recovery threshold is known exactly, as well as the optimal recovery quality at each value of $\lambda$ \cite{mmse-lowrank,mi,mi-proof,lelarge-limits-lowrank}. (These results actually consider a variant of the sparse Rademacher prior where the entries of the spike are \iid but we believe this does not change the information-theoretic limits of the problem.) See Figure~\ref{fig:sparse-rad} for a comparison of bounds for $d=2$ sparse Rademacher.

In Appendix~\ref{app:replica}, we present non-rigorous calculations through the replica method, predicting the precise location of the phase transition for each $d$, in the spherical and Rademacher cases. We have high confidence in the correctness of these predictions since replica predictions have been rigorously shown to be correct in various related settings (e.g.\ \cite{tal06parisi,tal06,mi,mi-proof,lelarge-limits-lowrank}). For the Rademacher prior, it can be deduced from \cite{km} that the replica prediction is a rigorous upper bound on the true threshold; see Appendix~\ref{app:replica}.

Our upper bounds for discrete priors (Rademacher and sparse Rademacher) are obtained by straightforward analysis of the MLE (maximum likelihood estimator), i.e.\ the non-efficient procedure that enumerates all possible spikes and tests which one is most likely. Our upper bound for the spherical prior showing strict separation from the injective norm, is proven by harnessing known properties of the $d = 2$ case (eigenvalue transition in spiked matrices). Our lower bound techniques are discussed in the next section.

The rest of the paper is organized as follows. In Section~\ref{sec:bounds} we present the main tools used in our lower bounds, including the statement of our main lower bound theorem (Theorem~\ref{thm:noise-conditioning}) along with a sketch of its proof using our noise conditioning method. In Section~\ref{sec:sph} we prove our lower and upper bounds for the spherical prior (assuming Theorem~\ref{thm:noise-conditioning}). In Section~\ref{sec:discrete} we prove our lower and upper bounds for discrete priors (again assuming Theorem~\ref{thm:noise-conditioning}), including the Rademacher and sparse Rademacher priors. In Section~\ref{sec:noise-conditioning-proof} we prove Theorem~\ref{thm:noise-conditioning}. Some results are deferred to the appendix, including the replica calculations for the spherical and Rademacher priors (Appendix~\ref{app:replica}).

\section{Lower bound techniques}\label{sec:bounds}

\subsection{$\chi^2$-divergence, contiguity, and non-recovery}
\label{sec:contig}

\paragraph{$\chi^2$-divergence and TV distance}
For probability distributions $P$ and $Q$, with $P$ absolutely continuous with respect to $Q$, the $\chi^2$-divergence is defined as
$$ \chi^2(P \dmid Q) = \Ex_{Q}\left[ \left(\dd[P]{Q}\right)^2 \right] - 1. $$
When $P,Q$ are continuous distributions with densities $p,q$ respectively, note that this is simply
$$ \chi^2(P \dmid Q) = \Ex_{y \sim Q}\left[ \left(\frac{p(y)}{q(y)}\right)^2 \right] - 1.$$

Let $P_n = \WT_n(d,\lambda,\cX)$ be the spiked ensemble, and let $Q_n = \WT_n(d)$ be its unspiked analogue. Our lower bounds will proceed by bounding $\chi^2(P_n \dmid Q_n)$ (or slight variants) and drawing various conclusions from the value.

For tensors of order $d \geq 3$, we will bound the TV (total variation) distance via the inequality
\begin{equation} \TV(P_n,Q_n) \leq 2 \sqrt{\chi^2(P_n \dmid Q_n)}. \label{eq:chi2-pinsker}\end{equation}
In particular, if we can establish that $\chi^2(P_n \dmid Q_n) = o(1)$ as $n\to\infty$, it follows that $\TV(P_n,Q_n) = o(1)$, implying that weak detection is impossible.

\paragraph{$\chi^2$-divergence and contiguity} For matrices ($d=2$), we cannot hope to show $\TV(P_n,Q_n) = o(1)$ because weak detection is possible for all $\lambda > 0$ using the trace of the matrix. We instead show lower bounds against strong detection, using the notion of \emph{contiguity} introduced by Le~Cam \cite{lecam}.

\begin{definition}[\cite{lecam}]
Let $P_n$, $Q_n$ be sequences of distributions defined on the measurable space $(\Omega_n,\F_n)$. We say that $P_n$ is \defn{contiguous} to $Q_n$, and write $P_n \contig Q_n$, if for any sequence $A_n$ of events with $Q_n(A_n) \to 0$ as $n \to \infty$, we also have $P_n(A_n) \to 0$ as $n \to \infty$.
\end{definition}

\noindent Note that $P_n \contig Q_n$ implies that strong detection is impossible; to see this, suppose we had a distinguisher and let $A_n$ be the event that it outputs `$P_n$' to arrive at a contradiction. (The definition of contiguity is asymmetric but contiguity in either direction implies that strong detection is impossible.)

The following second moment method connects the $\chi^2$-divergence with contiguity.

\begin{lemma}[see e.g.\ \cite{mrz,bmvvx}]
\label{lem:sec}
If $\chi^2(P_n \dmid Q_n) = O(1)$ as $n \to \infty$, then $P_n \contig Q_n$.
\end{lemma}
\noindent This is referred to as a \emph{second moment method} due to the ``second moment'' $\EE_{Q_n}\left(\dd[P_n]{Q_n}\right)^2$ appearing in the definition of $\chi^2$-divergence.
\begin{proof}
Let $A_n$ be a sequence of events. Using Cauchy--Schwarz,
\begin{align*}
P_n(A_n) 
= \int_{A_n} \dd[P_n]{Q_n} \,\dee Q_n \le \sqrt{\int_{A_n} \left( \dd[P_n]{Q_n} \right)^2 \,\dee Q_n} \;\cdot\; \sqrt{\int_{A_n} \,\dee Q_n}. 
\end{align*}
The bound on $\chi^2$-divergence implies that the first factor on the right-hand side is bounded. This means if $Q_n(A_n) \to 0$ then also $P_n(A_n) \to 0$.
\end{proof}

To summarize, we now know that if the $\chi^2$-divergence is $o(1)$ then weak detection is impossible, and if the $\chi^2$-divergence is $O(1)$ then strong detection is impossible.

We remark that computing $\chi^2(Q_n \dmid P_n)$ seems significantly harder than computing $\chi^2(P_n \dmid Q_n)$ (where $P_n$ is spiked and $Q_n$ is unspiked). Establishing contiguity in the opposite direction $Q_n \contig P_n$ typically requires additional work such as the small subgraph conditioning method (see e.g.\ \cite{mns,bmnn,wor-survey}). Our methods will only require us to compute $\chi^2$-divergence in the `easy' direction.

\paragraph{$\chi^2$-divergence and conditioning}

It turns out that in some cases, the $\chi^2$-divergence can be dominated by extremely rare `bad' events, causing it to be unbounded even though $P_n \contig Q_n$. Towards fixing this, a key observation is that if we replace $P_n$ by a modified distribution $\tilde P_n$ that only differs from $P_n$ with probability $o(1)$ as $n\to\infty$ (i.e.\ $\TV(P_n,\tilde P_n) = o(1)$), this does not affect whether or not the detection and recovery problems can be solved. Therefore, by choosing $\tilde P$ to condition on the high-probability `good' events we can hope to make $\chi^2(\tilde P_n \dmid Q_n)$ controlled even in some cases when $\chi^2(P_n \dmid Q_n)$ is not. If we can show $\chi^2(\tilde P_n \dmid Q_n) = O(1)$ it follows that $\tilde P_n \contig Q_n$, which implies $P_n \contig Q_n$ as desired. Prior work (e.g.\ \cite{bmnn,bmvx,pwbm}) has applied this conditioning technique to `good' events depending on the spike $x$; specifically, $\tilde P_n$ enforces that the entries of $x$ occur in close-to-typical proportions. As the main technical novelty of the current work, we introduce a technique that we call \emph{noise conditioning} where we condition on `good' events that depend jointly on the spike $x$ and the noise $W$. Specifically, we require $\langle W,x^{\otimes d} \rangle$ to have close-to-typical value. This method will allow us to obtain asymptotically tight lower bounds in cases where prior work \cite{mrz,bmvx} has been loose by a constant factor of $\sqrt{2}$. The noise conditioning method is discussed further in Section~\ref{sec:noise-cond}.

\paragraph{$\chi^2$-divergence and non-recovery}

First consider the case $d \ge 3$ where we are able to show $\TV(P_n,Q_n) = o(1)$ as $n\to\infty$. It follows immediately that weak recovery is impossible, provided that weak recovery is impossible in the unspiked model $Q_n$. Weak recovery is of course impossible in $Q_n$ for any reasonable spike prior, and if we had an `unreasonable' prior then we should not have had $\TV(P_n,Q_n) = o(1)$ in the first place. Our non-recovery proof in Section~\ref{sec:pf-non-recovery} makes this precise, showing that for the spiked tensor model, if $\TV(P_n,Q_n) = o(1)$ then weak recovery is impossible.

In the case $d=2$, non-recovery is less straightforward. Various works have forged a connection between bounded $\chi^2$-divergence and non-recovery \cite{bmnn,bmvx}. For a wide class of problems with additive Gaussian noise, Theorem~4 of \cite{bmvx} implies that if the $\chi^2$-divergence is $O(1)$ then weak recovery is impossible. Unfortunately their result cannot be immediately applied in our setting because, due to our noise conditioning technique, we do not exactly have an additive Gaussian noise model. We will leave the proof of non-recovery for $d=2$ for future work, although we strongly believe it to be true (in the same regime where strong detection is impossible).

\subsection{Large deviation rate functions}

Our lower bounds on tensor PCA problems will depend on the prior through tail probabilities of the correlation $\langle x,x' \rangle$ of two independent spikes drawn from the prior; note that this quantity is typically of order $1/\sqrt{n}$ but may be as large as $1$. We require detailed tail information, which we summarize in two objects: a \emph{rate function}, describing the asymptotic behavior of deviations of order $1$, and a \emph{local subgaussian condition} \cite{chareka2006}, bounding the deviations of size $[0,\eps]$ non-asymptotically.

First we define the rate function $f_\cX$ corresponding to a prior $\cX$. Intuitively, $f_\cX:[0,1) \to [0,\infty)$ is roughly defined by $\prob{\langle x,x' \rangle \geq t} \approx \exp(-n f_\cX(t))$. More formally, one should think of $f_\cX(t)$ as being equal to $- \lim_{n \to \infty} \frac{1}{n} \log \prob{\langle x,x' \rangle \geq t}$ but we will not technically require this in our definition since it will be ok to use a weaker (smaller) rate function than the `true' one (although we will always use the `true' one in our examples). For technical reasons, we formally define the rate function as follows.

\begin{definition}\label{def:rate-function}
Let $\cX = \{\cX_n\}$ be a prior supported on unit vectors in $\RR^n$. For $x,x'$ drawn independently from $\cX_n$ and $t \in [0,1)$, let
$$ f_{n,\cX}(t) = -\frac1n \log \prob{\langle x,x' \rangle \geq t}. $$
Suppose we have $f_{n,\cX}(t) \geq b_{n,\cX}(t)$ for some sequence of functions $b_{n,\cX}$ that converges uniformly on $[0,1)$ to $f_\cX$ as $n \to \infty$. Then we call such $f_{\cX}$ the \emph{rate function} of the prior $\cX$.
\end{definition}

\begin{remark}
\label{rem:tail-unif}
The condition $f_{n,\cX} \ge b_{n,\cX} \to f_\cX$ (uniformly) is satisfied if we have a tail bound of the form $\prob{\langle x,x' \rangle \ge t} \le \mathrm{poly}(n) \exp(-n f_\cX(t))$.
\end{remark}

\noindent We will assume that the distribution of $\langle x,x' \rangle$ is symmetric about zero, so that only the upper tail bound is required. Note the following basic properties of the rate function (which we can assume without loss of generality): $f_\cX(0) = 0$ and $f_\cX$ is monotone increasing. We will see that for continuous priors (such as the spherical prior), the rate function blows up to infinity at $t=1$, whereas for discrete priors (such as Rademacher) it remains finite at $t=1$.

We now present the local subgaussianity condition which gives tighter control of the small deviations.

\begin{definition}\label{def:locally-subg}
Let $\cX$ be a prior supported on unit vectors in $\RR^n$. We say that $\cX$ is \emph{locally subgaussian with constant $\sigma^2$} if for all $\eta > 0$, there exists $T > 0$ with
$$ \Pr_{x,x' \sim \cX}[\langle x,x' \rangle \ge t] \lesssim \exp\left( -\frac{n t^2}{2\sigma^2+\eta} \right) \qquad \forall t \in [0,T]. $$
\end{definition}
\noindent Here $A \lesssim B$ means there exists a constant $C$ such that $A \le CB$. (Above, $C$ is allowed to depend on $\eta$ but not $n$.) Note that if $\langle x,x' \rangle$ is $(\sigma^2/n)$-subgaussian, then $\cX$ is locally subgaussian with constant $\sigma^2$. However, the converse is false: for instance, the sparse Rademacher prior with sufficiently small density $\rho$ has a strictly better local subgaussian constant than its subgaussian constant.

\subsection{Main lower bound theorem}

The following theorem encapsulates our lower bound result:

\begin{theorem}\label{thm:noise-conditioning}
Let $d \geq 2$, and let $\cX$ be a prior supported on unit vectors in $\RR^n$. Suppose the following holds for some $\lambda^* > 0$.
\begin{enumerate}[label=(\roman*)]
    \item The distribution of $\langle x,x' \rangle$ (where $x,x'$ are drawn independently from $\cX$) is symmetric about zero,
    \item $\cX$ has a rate function $f_\cX$,
    \item $f_\cX(t) \ge \frac{(\lambda^*)^2}{2} \frac{t^d}{1+t^d}$ for all $t \in [0,1)$,
    \item $\cX$ is locally subgaussian with some constant $\sigma^2$,
    \item if $d=2$, $\lambda^* \leq 1/\sigma$.
\end{enumerate}
Then for all $\lambda < \lambda^*$:
\begin{itemize}
    \item $\WT(d,\lambda,\cX)$ is contiguous to $\WT(d)$, so no hypothesis test to distinguish them achieves full power (i.e.\ strong detection is impossible),
    \item if $d \ge 3$, $\TV(\WT(d,\lambda,\cX),\WT(d)) = o(1)$, so every hypothesis test has asymptotically zero power (i.e.\ weak detection is impossible),
    \item if $d \ge 3$, no estimator $\hat x = \hat x(T)$, when given a sample $T \sim \WT(d,\lambda,\cX)$, achieves expected correlation $\EE |\langle x,\hat x \rangle|$ (where $x$ is the true spike) bounded above $0$ as $n \to \infty$ (i.e.\ weak recovery is impossible).
\end{itemize}
\end{theorem}
\noindent The proof is deferred to Section~\ref{sec:noise-conditioning-proof}. The essential condition is (iii) and the others exist for technical reasons. We remark that condition (v) is simply a slight strengthening of condition (iii) in the following sense.

\begin{proposition}
\label{prop:chernoff-subg}
Let $d = 2$. Suppose the rate function $f_\cX$ admits a local Chernoff-style tail bound of the following form: there exists $T > 0$ such that
$$\prob{\langle x,x' \rangle \ge t} \lesssim \exp(-nf_\cX(t)) \qquad \forall t \in [0,T].$$
Then condition (iii) implies conditions (iv) and (v).
\end{proposition}
\noindent Note that this is stronger than Remark~\ref{rem:tail-unif} because there is no $\mathrm{poly}(n)$ factor.
\begin{proof}
Fix $\eta > 0$. Let $\eps > 0$, to be chosen later. For all $t \in [0,T']$ we have $\frac{t^2}{1+t^2} \ge \frac{t^2}{1+T'^2}$ and so there exists $T' > 0$ such that $\frac{t^2}{1+t^2} \ge (1-\eps)t^2$ for all $t \in [0,T']$. For $t \in [0,\min(T,T')]$,
$$\prob{\langle x,x' \rangle \ge t} \lesssim \exp(-n f_\cX(t)) \le \exp\left(-n\, \frac{(\lambda^*)^2}{2} \frac{t^2}{1+t^2}\right) \le \exp\left(-\frac{nt^2}{2} (\lambda^*)^2(1-\eps)\right).$$
Choose $\eps$ small enough so that this is at most $\exp\left(-\frac{nt^2}{2/(\lambda^*)^2 + \eta}\right)$, the local subgaussian condition.
\end{proof}

\subsection{Proof idea: noise conditioning}
\label{sec:noise-cond}

The proof of our lower bounds hinge on a simple ``noise conditioning'' modification to the second moment method ($\chi^2$-divergence) for proving contiguity. We have found that this method yields asymptotically tight bounds in various regimes where the basic second moment method is loose by a factor of $\sqrt{2}$. (Below we will see how the number $\sqrt{2}$ arises.) We expect that the same idea could be used to close the various other $\sqrt{2}$-factor gaps that exist for contiguity arguments in prior work: the sparse stochastic block model in the regime of a constant number of equally-sized communities with low signal-to-noise \cite{bmnn}; submatrix localization in the limit of a large constant number of blocks \cite{bmvx}; Gaussian mixture clustering with a large constant number of clusters \cite{bmvx}; and synchronization over a finite group of large constant size (with either truth-or-Haar noise or Gaussian noise on all frequencies) \cite{pwbm}. However, we leave investigation of these other problems for future work\footnote{The latest version of \cite{bmvvx} closes the asymptotic gaps  for sparse PCA and submatrix localization using a different modification of the second moment method; the problem for Gaussian mixtures remains open, to our knowledge.}.

The following is a proof sketch of our main lower bound theorem (Theorem~\ref{thm:noise-conditioning}) using the noise conditioning method. As discussed in Section~\ref{sec:contig}, if the second moment $\EE_{Q_n} \left(\dd[P_n]{Q_n}\right)^2$ is bounded as $n \to \infty$ then $P_n$ is contiguity to $Q_n$ and so it is impossible to reliably distinguish the two distributions. Taking $P_n$ to be the spiked tensor model and $Q_n$ the corresponding unspiked model, the second moment can be computed to be
\begin{equation}
\label{eq:M}
\Ex_{Q_n}\left(\dd[P_n]{Q_n}\right)^2 = \Ex_{x,x' \sim \cX} \,\Ex_{T \sim Q_n} \exp\left(\frac{n\lambda}{2} \langle T, x^{\otimes d} + x'^{\otimes d} \rangle - \frac{n\lambda^2}{2} \right)
\end{equation}
where $x,x'$ are drawn independently from the prior $\cX$. The standard approach used by previous work is to next apply the Gaussian moment-generating function to compute the expectation over $T \sim Q_n$ in closed form, yielding
\begin{equation}
\label{eq:M-uncond}
\Ex_{x,x' \sim \cX} \exp\left(\frac{n\lambda^2}{2} \beta\right)
\end{equation}
where $\beta = \langle x,x' \rangle^d$. Our approach is instead to use the fact that (as discussed in Section~\ref{sec:contig}) we can change $P_n$ to $\tilde P_n$ that excludes low-probability bad events, without affecting contiguity. Specifically, we take $\tilde P_n$ to condition on (informally) $\langle T, x^{\otimes d} \rangle \approx \lambda$, which is a high-probability event under $P_n$. Now, instead of (\ref{eq:M}), the second moment can be computed to be
\begin{align*}
\Ex_{Q_n}\left(\dd[\tilde P_n]{Q_n}\right)^2 &\approx \Ex_{x,x' \sim \cX} \,\Ex_{T \sim Q_n} \one_{\langle T,x^{\otimes d} \rangle \approx \lambda}\one_{\langle T,x'^{\otimes d} \rangle \approx \lambda}\exp\left(\frac{n\lambda}{2} \langle T, x^{\otimes d} + x'^{\otimes d} \rangle - \frac{n\lambda^2}{2} \right) \\
&\approx \Ex_{x,x' \sim \cX} \exp\left(\frac{n\lambda}{2} (2\lambda) - \frac{n\lambda^2}{2} \right) \Pr_{T \sim Q_n}\left[{\langle T,x^{\otimes d} \rangle \approx \langle T,x'^{\otimes d} \rangle \approx \lambda}\right], \\
\end{align*}
which can be computed to be
\begin{equation}
\label{eq:M-cond}
\Ex_{x,x' \sim \cX} \exp\left(\frac{n\lambda^2}{2} \frac{\beta}{1+\beta}\right)
\end{equation}
where again $\beta = \langle x,x' \rangle^d$. Note that when the contribution from $\beta = 1$ dominates, this gives a factor of $\sqrt{2}$ advantage on $\lambda$ compared to (\ref{eq:M-uncond}). This explains the gaps of $\sqrt{2}$ present in prior work, and allows us to close them.

The final step is to bound (\ref{eq:M-cond}) using the rate function $f_\cX$ of $\cX$. Roughly, the rate function gives us the tail bound $\prob{\langle x,x' \rangle \ge t} \approx \exp(-n f_\cX(t))$. To use this, we write (\ref{eq:M-cond}) as a tail bound integral and then apply a change of variables:
\begin{align*}
\Ex_{x,x' \sim \cX} \exp\left(\frac{n\lambda^2}{2} \frac{\beta}{1+\beta}\right)
&= \int_0^\infty \problr{\exp\left(\frac{n\lambda^2}{2} \frac{\beta}{1+\beta}\right) \ge u} \dee u \\
&\simeq \int_0^1 \prob{\langle x,x' \rangle \ge t} \exp\left(\frac{n\lambda^2}{2} \frac{t^d}{1+t^d}\right)\frac{n\lambda^2}{2} \frac{dt^{d-1}}{(1+t^d)^2}\, \dee t \\
\intertext{where $u = \exp\left(\frac{n \lambda^2}{2} \frac{t^d}{1+t^d}\right)$,}
&\approx \int_0^1 \exp\left[n\left( - f_\cX(t) + \frac{\lambda^2}{2} \frac{t^d}{1+t^d}\right)\right] \dee t.
\end{align*}
Note that this would be bounded as $n \to \infty$ if we had $- f_\cX(t) + \frac{\lambda^2}{2} \frac{t^d}{1+t^d} < 0$ for all $t \in [0,1]$, which gives condition (iii) in Theorem~\ref{thm:noise-conditioning}. Note however that this cannot be satisfied at $t = 0$ (since $f_\cX(t) = 0$), so our proof will require extra work in order to handle the $t \in [0,\eps)$ interval. This is where the local subgaussianity condition will be used.

\section{Results for the spherical prior}
\label{sec:sph}

In this section we discuss lower and upper bounds for the \emph{spherical prior} $\cXs$, the uniform prior on the unit sphere.

\subsection{Rate function}

\begin{lemma} The spherical prior $\cXs$ has rate function $f_\cXs(t) = -\frac12 \log(1-t^2)$ and is locally subgaussian. \end{lemma}
\begin{proof}
Note that with $x,x' \sim \cXs$, $\frac12 + \frac12 \langle x,x' \rangle \sim \Beta(\frac{n}{2},\frac{n}{2})$. It follows from subgaussianity of the beta distribution \cite{sam} that $\langle x,x' \rangle$ is $O(1/n)$-subgaussian, so that $\cXs$ is locally subgaussian (with some constant). Moreover, note that
$$ -\frac1n \log \Pr[\langle x,x' \rangle \geq t] = -\frac1n \log I_{(1-t)/2}\left( \frac{n}{2},\frac{n}{2} \right), $$
where $I$ denotes the incomplete beta function. By equation 8.18.9 from \cite{dlmf}, this converges uniformly to $-\frac12 \log(1-t^2)$ on $[0,1)$. Thus we satisfy the convergence conditions in Definition~\ref{def:rate-function}, taking $b_{n,\cX} = f_{n,\cX}$.
\end{proof}

\subsection{Lower bound}

We can now complete the lower bound portion of Theorem~\ref{thm:dinf-spherical}. From the discussion above together with Theorem~\ref{thm:noise-conditioning}, we have the following.
\begin{theorem}
\label{thm:sph-lower}
Consider the spherical prior $\cX_\sph$ with $d\ge 3$. Let $\lambda^*_{\sph,d}$ be the supremum of all values $\lambda^*$ for which
$$\frac{(\lambda^*)^2}{2}\frac{t^d}{1+t^d} \leq -\frac12 \log(1-t^2) \qquad \forall t \in [0,1).$$
Then weak detection and weak recovery are impossible for all $\lambda < \lambda^*_{\sph,d}$.
\end{theorem}
\noindent Note that the $d=2$ case is already well-understood: strong detection and weak recovery are possible when $\lambda > 1$ and impossible when $\lambda < 1$ \cite{mrz,pwbm}, matching the eigenvalue transition. (Weak recovery is possible for all $\lambda > 0$ simply by inspecting the trace of the matrix.) In Appendix~\ref{app:asymptotics} we establish the following asymptotics as $d\to\infty$:
$$(\lambda_{\sph,d}^*)^2 = 2 \log d + 2 \log \log d + 2 - 4 \log 2 + o(1).$$

\subsection{Upper bound}

In this section we prove an upper bound showing that for sufficiently large $\lambda$, detection and recovery are possible in the spherically-spiked tensor model. Indeed, our upper bound will apply to any prior supported on the unit sphere. Recall that the injective norm of a $d$-tensor is defined as $\|T\| = \max_{\|x\|=1} \langle x^{\otimes d},T \rangle$. As noted by \cite{rm}, as soon as the injective norm $\lambda$ of the spike exceeds that of the noise, detection and recovery are possible. We improve on this bound, showing that the injective norm of the spiked model exceeds that of the unspiked model at a lower threshold. This is achieved by studying how perturbations away from the spike can line up with fluctuations in the noise to achieve a large injective norm.

For each $d$, the precise limit value (in probability, as $n \to \infty$) of $\|\WT(d)\|$ has been non-rigorously computed using the replica method \cite{cs92}; following \cite{rm}, we refer to this limit value as $\mu_d$ (although our normalization differs from theirs). This was later proved rigorously, first only for even values of $d$ \cite{tal06,abac} (see \cite{rm} for a summary) and later for all $d$ \cite{subag-pspin}.

\begin{theorem}[\cite{cs92,tal06,abac,rm,subag-pspin}]
\label{thm:inj-norm}
Fix $d \ge 3$. Define
$\mu_d = x \sqrt{2/d} $ 
where $x \ge 2\sqrt{d-1}$ is the unique solution to
\begin{equation}\frac{2-d}{d} - \log\left(\frac{dz^2}{2}\right) + \frac{d-1}{2} z^2 - \frac{2}{d^2 z^2} = 0, \qquad z(x) = \frac{1}{(d-1)\sqrt{2d}} \left(x - \sqrt{x^2 - 4(d-1)}\right). \label{eq:injective-eq}\end{equation}
With probability $1-o(1)$ as $n\to\infty$ we have
$$\|\WT(d)\| = \mu_d + o(1).$$
\end{theorem}

Now suppose $T = \lambda x^{\otimes d} + W$ is a spiked $d$-tensor. Clearly if $\lambda > \mu_d$ then we will have $\|T\| > \mu_d$ with probability $1-o(1)$ and so strong detection is possible by thresholding the injective norm. However, recall that for matrices ($d=2$) this is not tight: $\mu_2 = 2$ yet detection is possible for all $\lambda > 1$ (see Theorem~\ref{thm:bbp-wig}). Our result will imply that for all $d$ it remains true that the detection threshold $\lambda^*$ is strictly less than $\mu_d$, but the gap between them shrinks as $d \to \infty$. The argument simply harnesses the $d=2$ result in a ``black box'' fashion (which, perhaps surprisingly, gives quite a good upper bound; see Figure~\ref{fig:compare}). We will first give a lower bound on the injective norm of a spiked tensor, stronger than the trivial lower bound of $\lambda$. This easily implies results for strong detection and weak recovery (Corollary~\ref{cor:sph-upper} below).

\begin{theorem}
\label{thm:sph-upper}
Fix $d \ge 3$ and $\lambda \ge 0$. Let $\cX$ be any prior supported on the unit sphere in $\RR^d$. Let $T = \lambda x^{\otimes d} + W$ be a spiked tensor drawn from $\WT(d,\lambda,\cX)$. With probability $1-o(1)$ as $n\to\infty$ we have
$$\|T\| \ge L_d(\lambda) - o(1)$$
where
\begin{equation} L_d(\lambda) = \max_{m \in [0,1]} m^d \left(\lambda + \sqrt{\frac{2d}{d-1}} \sqrt{M (1+M)}\right), \qquad M(m) = (d-1) \frac{1-m^2}{m^2}. \label{eq:upper-opt}\end{equation}
\end{theorem}

\begin{proof}
By Gaussian spherical symmetry we can assume without loss of generality that the spike is $x = e_1$, the first standard basis vector; thus we write $T = \lambda e_1^{\otimes d} + W$ with $W \sim \WT(d)$. Writing a general unit vector as $v = m e_1 + \sqrt{1-m^2} u$, where $u \perp e_1$ is a unit vector, we expand $v^{\otimes d}$ to write
$$ \langle T, v^{\otimes d} \rangle = \lambda m^d + \sum_{i=0}^d \binom{d}{i} m^{d-i} (1-m^2)^{i/2} \langle W_i, u^{\otimes i} \rangle, $$
where $W_i$ is the sub-tensor of $W$ for which the first $i$ indices are not $1$ and the remaining indices are $1$.

Note that the $W_i$ are independent. We will choose $u$ to optimize the interactions with $W_1$ and $W_2$, independently from all other $W_i$. It follows that for all $i \ne 1,2$, the terms $\langle W_i, u^{\otimes i} \rangle$ are $o(1)$ as $n\to\infty$ with high probability, and we have
\begin{equation} \langle T, v^{\otimes d} \rangle = \lambda m^d + d m^{d-1} \sqrt{1-m^2} \langle W_1, u \rangle + \binom{d}{2} m^{d-2} (1-m^2) u^\top W_2 u + o(1). \label{eq:T-expand}\end{equation}

Specifically, let us take $u$ to be the top eigenvector of the auxiliary spiked matrix
$$ Y = \beta \frac{d}{2} W_1 W_1^\top + \sqrt{\frac{d(d-1)}{2}} W_2 $$
for some parameter $\beta \geq 1$ to be optimized later. Since $\sqrt{\frac{d}{2}} W_1$ has norm converging to 1 in probability and $\sqrt{\frac{d(d-1)}{2}} W_2$ is a Gaussian Wigner matrix, the classical eigenvalue transition for spiked Wigner matrices (Theorem~\ref{thm:bbp-wig}) implies that the top eigenvalue $u^\top Y u$ converges almost surely to $\beta + 1/\beta$ (as $n \to \infty$), and the top eigenvector $u$ is correlated with the normalized spike as $\langle u, W_1 \sqrt{d/2} \rangle^2 \to 1 - 1/\beta^2$ almost surely \cite{peche,fp,cdf,eig-vec}. Noting that
$$\beta + \frac1\beta + o(1) = u^\top Y u = \frac{\beta d}{2} \langle u,W_1 \rangle^2 + \sqrt{\frac{d(d-1)}{2}} \,u^\top W_2 u, $$
we solve for $u^\top W_2 u = \sqrt{\frac{2}{d(d-1)}}\cdot\frac{2}{\beta} + o(1)$ and plug into (\ref{eq:T-expand}) to obtain
$$ \|T\| \ge \langle v^{\otimes d}, T \rangle = \lambda m^d + m^{d-1} \sqrt{2d(1-m^2)(1-1/\beta^2)} + m^{d-2} (1-m^2) \sqrt{2d(d-1)} /\beta + o(1). $$

\noindent To obtain the strongest possible bound, we maximize this expression over all $m \in [0,1]$ and $\beta \ge 1$. We can optimize in closed form for
$$ \beta = \left( 1 + \frac{m^2}{(1-m^2) (d-1)}\right)^{1/2}. $$
The result as stated now follows by simple algebra.
\end{proof}

\begin{corollary}
\label{cor:sph-upper}
Fix $d \ge 3$ and $\lambda \ge 0$. Let $\cX$ be any prior supported on the unit sphere in $\RR^d$. Let $\mu_d$ be given by Theorem~\ref{thm:inj-norm} and let $L_d(\lambda)$ be defined as in Theorem~\ref{thm:sph-upper}. If $L_d(\lambda) > \mu_d$ then strong detection and weak recovery are possible for $\WT(d,\lambda,\cX)$.
\end{corollary}

\begin{proof}
If $T = \lambda x^{\otimes d} + W$ we have $\|T\| \ge L_d(\lambda) - o(1)$ with probability $1-o(1)$ as $n\to\infty$. If instead $T = W$ is unspiked, we have $\|T\| \le \mu_d + o(1)$. It follows that strong detection is possible by thresholding $\|T\|$.

To perform weak recovery given $T = \lambda x^{\otimes d} + W$, output the unit vector $v$ maximizing $\langle v^{\otimes d},T \rangle$. With probability $1-o(1)$ we have
$$L_d(\lambda) - o(1) \le \|T\| = \langle v^{\otimes d},T \rangle = \lambda \langle v,x \rangle^d + \langle v^{\otimes d},W \rangle \le \lambda \langle v,x \rangle^d + \mu_d + o(1)$$
and so $\langle v,x \rangle^d \ge \frac{1}{\lambda} (L_d(\lambda) - \mu_d) - o(1)$, which is bounded above 0 as $n \to \infty$ (by assumption). Therefore weak recovery is possible.
\end{proof}

To prove the upper bound portion of Theorem~\ref{thm:dinf-spherical}, we let $\Lambda^*_{\sph,d}$ be the point at which $\lambda > \Lambda^*_{\sph,d}$ implies $L_d(\lambda) > \mu_d$. Note that by considering $m=1$ in Theorem~\ref{thm:sph-upper}, we obtain that the injective norm of a spiked tensor is at least the size of the spike (i.e.\ $L_d(\lambda) \ge \lambda$), as noted in \cite{rm}. However, the derivative of (\ref{eq:upper-opt}) at $m=1$ is $-\infty$, implying a strict separation: $L_d(\lambda) > \lambda$ and so $\Lambda^*_{\sph,d} < \mu_d$. Therefore, for any $d$ it is possible to distinguish the spiked and unspiked models for some $\lambda$ strictly less than the injective norm of a random tensor.

In Appendix~\ref{app:asymptotics} we determine asymptotic forms of Theorem~\ref{thm:inj-norm} and Corollary~\ref{cor:sph-upper}. Specifically, the value $\mu_d$ behaves as $d \to \infty$ as
\begin{align*}
    \mu_d^2 &= 2 \log d + 2 \log \log d + 2 + o(1),
    \intertext{while
    $L_d(\lambda) > \mu_d$ as soon as $\lambda > \Lambda_{\sph,d}^*$, with}
    (\Lambda_{\sph,d}^*)^2 &= 2 \log d + 2 \log \log d + o(1).
\end{align*}

\section{Results for discrete priors}
\label{sec:discrete}

In contrast to the upper and lower bounds discussed in the previous section for the spherical prior, which diverge with $d$, the statistical thresholds for discrete priors tend to remain bounded as $d\to\infty$. This dichotomy is somewhat surprising given that the spherical and Rademacher thresholds agree exactly ($\lambda=1$) for $d=2$. In this section we present some general and some specific results confirming this phenomenon. In particular, we establish upper and lower bounds under the Rademacher and sparse Rademacher priors, which match in either of the limits $d \to \infty$ and $\rho \to 0$ (recall that $\rho$ is the density of nonzeros). We thus establish Theorems~\ref{thm:dinf-rademacher}, \ref{thm:dinf-sparse-rademacher}, and~\ref{thm:p0-sparse-rademacher}.

\subsection{Upper bounds}

A natural approach to hypothesis testing is the MLE (maximum likelihood estimate), the maximum of the tensor form over the support of the prior. For discrete priors of exponential cardinality (including Rademacher and sparse Rademacher), we can control this via a na\"ive union bound to obtain positive results for both detection and recovery (similar to \cite{bmvx}). Note however that this only yields non-efficient procedures for detection and recovery, since they require exhaustive search over all possible spikes.
\begin{proposition}\label{prop:cardinality-upper-bound}
Let $\cX$ be a prior supported on the unit sphere in $\RR^n$ and with exponential support:
$$ \limsup_{n \to \infty} \frac{1}{n} \log |\supp \cX_n| = c. $$ 
For any $d\ge 2$, when $\lambda > 2\sqrt{c}$, there exists a hypothesis test distinguishing $\WT(d,\lambda,\cX)$ from $\WT(d)$ with probability $o(1)$ of error as $n\to\infty$ (i.e.\ strong detection is possible) and furthermore there exists an estimator that achieves nontrivial correlation with the spike (i.e.\ weak recovery is possible).
\end{proposition}

\begin{proof}
Given a tensor $T$, consider the maximum likelihood statistic
\begin{equation}
\label{eq:mle}
m = \max_{v \in \supp\cX} \langle T, v^{\otimes d} \rangle.
\end{equation}
We will show that thresholding $m$ suitably yields a hypothesis test that distinguishes the spiked and unspiked models with $o(1)$ probability of error of either type.

Suppose first that $T = \lambda x^{\otimes d} + W$ is drawn from the spiked model $\WT(d,\lambda,\cX)$ with spike $x$. Then
$$ m \geq \langle T, x^{\otimes d} \rangle = \lambda + \langle W, x^{\otimes d} \rangle = \lambda + \cN(0,2/n), $$
which, for any fixed $\eps > 0$, is greater than $\lambda - \eps$ with probability $1 - o(1)$.

On the other hand, suppose that $T$ is drawn from the unspiked model $\WT(d)$. Then by taking a union bound over all $v \in \supp \cX$,
\begin{align*}
    \Pr[m \geq \lambda - 2\eps] &\leq |\supp \cX_n| \Pr[\langle W, v^{\otimes d} \rangle \geq \lambda - 2\eps] \\
    &= |\supp \cX_n| \Pr[ \N(0,2/n) \geq \lambda - 2\eps ] \\
    &\leq \exp\left(n\left(\frac1n \log |\supp \cX_n| - \frac14 (\lambda - 2\eps)^2 \right)\right),
\end{align*}
and so long as $\lambda > 2\sqrt{c}$, we can choose $\eps > 0$ such that this probability of error is $o(1)$. Hence, thresholding $m$ at $\lambda - \eps$, we obtain a hypothesis test with $o(1)$ probability of error of either type.

In order to perform weak recovery given $T = \lambda x^{\otimes d} + W$, output the $v$ that maximizes (\ref{eq:mle}). From the analysis above we have with probability $1-o(1)$, $\langle T,v^{\otimes d} \rangle \ge \lambda - \eps$ and $\langle W,v^{\otimes d} \rangle \le \lambda - 2\eps$. This means we have
$$\lambda - \eps \le \langle T, v^{\otimes d}\rangle = \lambda \langle x,v \rangle^d + \langle W,v^{\otimes d} \rangle \le \lambda \langle x,v \rangle^d + \lambda - 2\eps$$
and so $\langle x,v \rangle^d \ge \eps/\lambda$, implying weak recovery.
\end{proof}

The following MAP (maximum a posteriori) result is stronger for non-uniform priors. Recall that the R\'enyi entropy of order $\alpha \ge 0$ for a discrete random variable $X$ (taking values $x_i$ with probability $p_i$) is $H_\alpha(X) = \frac{1}{1-\alpha} \log\left(\sum_i p_i^\alpha\right)$.
\begin{proposition}\label{prop:entropy-upper-bound}
Fix $d\ge 2$ and let $\cX$ be a prior on the unit sphere in $\RR^n$, with bounded R\'enyi entropy density: there exists $\delta > 0$ with 
$$ \limsup_{n \to \infty} \frac1n H_{1-\delta}(\cX) < \infty. $$
Suppose furthermore that the varentropy $V(\cX) \defeq \Var_{x \sim \cX}[-\log \Pr_\cX(x)]$ is $o(n^2)$.
Let $s = \limsup_{n \to \infty} \frac{1}{n} H_1(\cX)$, the Shannon entropy density. Then when $\lambda > 2\sqrt{s}$, there exists a hypothesis test distinguishing $\WT(d,\lambda,\cX)$ from $\WT(d)$ with probability $o(1)$ of error (i.e.\ strong detection is possible) and furthermore there exists an estimator that achieves nontrivial correlation with the spike (i.e.\ weak recovery is possible).
\end{proposition}
\noindent The proof is deferred to Appendix~\ref{app:proof-entropy-upper-bound} but the idea is the same as above. Note that as $H_\alpha$ is monotonically decreasing in $\alpha$, the R\'enyi entropy bound guarantees that the Shannon entropy density is finite.

\subsection{Lower bounds as $d \to \infty$}

As a corollary of our main lower bound theorem (Theorem~\ref{thm:noise-conditioning}) we have the following lower bound in the $d \to \infty$ limit.
\begin{proposition}\label{prop:renyi-lower-bound}
Suppose that $\cX$ has a rate function $f_\cX(t)$ (Definition~\ref{def:rate-function}) and is locally subgaussian (with some constant). Suppose the limit $F = \lim_{t \to 1^-} f_\cX(t)$ exists. Then for any $\lambda < 2 \sqrt{F}$, it holds for sufficiently large $d$ that weak detection and weak recovery are impossible.
\end{proposition}
\noindent Note that (assuming $f_\cX$ is continuous at $t=1$) we have $F = f_\cX(1) = - \lim_{n \to \infty} \frac{1}{n} \log(\Pr[ x = x' ])$ which is the collision entropy density.

\begin{proof}
By Theorem~\ref{thm:noise-conditioning}, it suffices to show that, for all $\lambda < 2\sqrt{F}$, it holds for sufficiently large $d$ that $\frac{\lambda^2}{2} \frac{t^d}{1+t^d} \le f_\cX(t)$ for all  $t \in [0,1)$. The intuition for the proof is that when $d$ is large, the function $\frac{t^d}{1+t^d}$ is essentially zero for most of the interval $t \in [0,1)$ but it jumps sharply up to $\frac{1}{2}$ near $t=1$; it follows that only the value of $f_\cX$ near $t=1$ is important.

Let $\eps > 0$, to be chosen later. Provided $\eps$ is sufficiently small, local subgaussianity guarantees that $f_\cX(t) \geq c t^2$ on $[0,\eps]$ for some $c > 0$; thus $\frac{\lambda^2}{2} \frac{t^d}{1+t^d} \le f_\cX(t)$ on $[0,\eps]$, for all $d \geq 3$.

We next establish that for $d$ sufficiently large we have $\frac{\lambda^2}{2} \frac{t^d}{1+t^d} \le f_\cX(t)$ for all $t \in [\eps,1-\eps]$. We have $f_\cX(\eps) > 0$ by local subgaussianity and so (by monotonicity of $f_\cX$) we have a positive lower bound on $f_\cX$ on $[\eps,1-\eps]$. By contrast, $\frac{\lambda^2}{2} \frac{t^d}{1+t^d}$ converges uniformly to $0$ on $[\eps,1-\eps]$ as $d\to\infty$; the claim follows.

It remains to show the bound on $[1-\eps,1)$. Here $f_\cX(t) \geq f_\cX(1-\eps)$, while $\frac{\lambda^2}{2} \frac{t^d}{1+t^d} \leq \frac{\lambda^2}{4}$ (since $x \mapsto \frac{x}{1+x}$ is increasing on $[0,1]$). It is therefore sufficient to choose $\eps$ small enough so that $f_\cX(1-\eps) \ge \frac{\lambda^2}{4}$. This is possible due to the assumption $\lambda^2 < 4 F$ along with the definition of $F$.
\end{proof}

\begin{remark}This argument goes through even if the collision entropy density $F$ is $\infty$, in which case we have that for any $\lambda$ there exists $d$ such that weak detection and weak recovery are impossible. In other words, the threshold $\lambda^*$ must necessarily go to infinity as $d\to\infty$, in contrast to the behavior $\lambda^* = O(1)$ for discrete priors. For instance, recall that the spherical prior has $F = \infty$ and indeed the threshold $\lambda^*$ diverges with $d$.
\end{remark}

\subsection{Rademacher prior}

Consider the Rademacher prior $\cX_\Rade$ where each entry of $x$ is drawn \iid from $\{\pm 1/\sqrt{n}\}$. In this section we prove Theorem~\ref{thm:dinf-rademacher} (large $d$ behavior) as well as give an explicit lower bound for each $d$. The upper bound is immediate from Proposition~\ref{prop:cardinality-upper-bound}.

The rate function of $\cX_\Rade$ follows from Cram\'er's theorem, with the rate function being the convex dual of the cumulant generating function $\log\cosh(t)$, which is
$$ f_{\cX_\Rade}(t) = \log 2 - H\left(\frac{1+t}{2}\right), $$
where $H$ is the binary entropy: $H(p) = - p\log p -(1-p)\log(1-p)$. The necessary convergence condition is satisfied with $b_{n,\cX_\Rade}(t) = f_{\cX_\Rade}(t)$, as the bound $f_{n,\cX_\Rade} \geq f_{\cX_\Rade}$ is precisely the Chernoff bound. Moreover, $\langle x,x' \rangle$ is an average of $n$ \iid symmetric Rademacher random variables, and thus is $1/n$-subgaussian by Hoeffding's theorem. Hence $\cX_\Rade$ is locally subgaussian with constant $1$.

The asymptotic lower bound for large $d$ now follows from Proposition~\ref{prop:renyi-lower-bound}. By Theorem~\ref{thm:noise-conditioning} we have the following explicit bound for any fixed $d$.

\begin{theorem}
\label{thm:rad-lower}
Consider the Rademacher prior $\cX_\Rade$ with $d\ge 3$. Let $\lambda^*_{\Rade,d}$ be the supremum of all values $\lambda^*$ for which
$$\frac{(\lambda^*)^2}{2}\frac{t^d}{1+t^d} \leq f_{\cX_\Rade}(t) \qquad \forall t \in [0,1).$$
Then weak detection and weak recovery are impossible for all $\lambda < \lambda^*_{\sph,d}$.
\end{theorem}
\noindent As with the spherical prior, the $d=2$ case is already well-understood: strong detection and weak recovery are possible only above the spectral threshold $\lambda = 1$ \cite{dam,pwbm}, and weak detection is possible for all $\lambda > 0$ by taking the trace.

\subsection{Sparse Rademacher prior}
\label{sec:sparse-rad}

Recall that the sparse Rademacher prior $\cX_\spRade(\rho)$ with sparsity $\rho \in (0,1]$ draws the spike $x$ by choosing at random a subset of exactly $\rho n$ out of the $n$ indices, setting these indices to $\pm 1/\sqrt{\rho n}$ at random, and setting the remaining indices to 0.

\begin{remark}
\label{rem:iid}
As in \cite{mi}, one can also define a similar \iid prior where each entry is drawn independently from the appropriate distribution: $\pm 1/\sqrt{\rho n}$ each with probability $\rho/2$, and 0 with probability $1-\rho$. We expect that the information theoretic limits are the same for the \iid variant, but in order to apply our lower bound argument, one should condition on the spike having close-to-typical sparsity anyway. For this reason, we consider the exact-support version in order to simplify our proof.
\end{remark}

We will see that the sparse Rademacher prior $\cX_\spRade(\rho)$ has rate function given by
$$f_\rho(t) = \min_{\zeta \in [\max(\rho t,2\rho-1),\rho]} G(\zeta) + \zeta f_\Rade\left(\frac{\rho t}{\zeta}\right)$$
where
$$G(\zeta) = -H(\{\zeta,\rho-\zeta,\rho-\zeta,1-2\rho+\zeta\}) + 2H(\rho)$$
and where $f_\Rade$ is the Rademacher rate function:
$$f_\Rade(t) = \log 2 - H\left(\frac{1+t}{2}\right).$$
Here $H(p) = -p \log p - (1-p)\log(1-p)$ is the binary entropy and $H(\{p_1,\ldots,p_k\}) = -\sum_{i=1}^k p_i \log p_i$.

The intuition behind this rate function is the following. The variable $\zeta$ represents the size (as a fraction of $n$) of the overlap of the supports of $x$ and $x'$. The ``entropy term'' $G(\zeta)$ captures the probability that a particular $\zeta$ value occurs: $\prob{\zeta} \approx \exp(-n G(\zeta))$. Once we have fixed $\zeta$, the problem reduces to a Rademacher problem, causing the $f_\Rade$ term to appear. Due to the exponential scaling, the rate function is dominated by a single value of $\zeta$, hence the minimization. Note that $[\max(\rho t,2\rho-1),\rho]$ is precisely the set of possible $\zeta$ values that can occur given that $\langle x,x' \rangle \ge t$.

\begin{remark}
One can similarly derive the rate function for any prior with \iid finitely-supported entries (say with support size $s$), conditioned on typical proportions of entries. As in Appendix~A of \cite{bmnn}, in general one needs to optimize over $s \times s$ matrices, and the entropy term takes a simple form in terms of the KL divergence.
\end{remark}

The following tail bound shows that $f_\rho$ satisfies our definition of rate function (see Remark~\ref{rem:tail-unif}).

\begin{proposition}
\label{prop:sparse-tail-weak}
Let $x,x'$ be drawn independently from the sparse Rademacher prior $\cX_\spRade(\rho)$. Then for any $t \in [0,1]$ we have
$$\prob{\langle x,x' \rangle \ge t} \lesssim n^{3/2} \exp(-n f_\rho(t)).$$
\end{proposition}

\noindent (Here $\lesssim$ means $\le$ up to a constant that may depend on $\rho$ but not on any other variable.) The proof is deferred to Appendix~\ref{app:sparse-rad-tail}. The next tail bound improves upon the above by eliminating the polynomial factor $n^{3/2}$; however, we only prove it for sufficiently small $t$. This tail bound implies the local subgaussian condition with the correct constant for $d=2$ (see Proposition~\ref{prop:chernoff-subg}).

\begin{proposition}
\label{prop:sparse-tail-strong}
For every $\rho \in (0,1]$ there exist constants $C,T > 0$ such that when $x,x'$ are drawn independently from $\cX_\spRade(\rho)$, we have
$$\prob{\langle x,x' \rangle \ge t} \le C \exp(-n f_\rho(t)) \qquad \forall t \in [0,T].$$
\end{proposition}

\noindent The proof is deferred to Appendix~\ref{app:sparse-rad-tail}. This completes the proof of Theorem~\ref{thm:dinf-sparse-rademacher} (large $d$ behavior), with the upper bound following from Proposition~\ref{prop:cardinality-upper-bound} and the lower bound following from Proposition~\ref{prop:renyi-lower-bound}. More explicitly, Theorem~\ref{thm:noise-conditioning} gives the following lower bound for each $d$.

\begin{theorem}
\label{thm:sparse-rad-lower}
Let $d\ge 2$. Consider the sparse Rademacher prior $\cX_\spRade(\rho)$. Let $\lambda^*_{\spRade(\rho),d}$ be the largest $\lambda^*$ for which
$$\frac{(\lambda^*)^2}{2} \frac{t^d}{1+t^d} \le f_\rho(t) \qquad \forall t \in [0,1).$$
If $\lambda < \lambda^*_{\spRade(\rho),d}$ then strong detection is impossible. If furthermore $d \ge 3$, then weak detection and weak recovery are also impossible.
\end{theorem}

We next consider the limit $\rho \to 0$, in which case the lower bound above has the following asymptotics.

\begin{proposition}
\label{prop:sparse-rad-p0}
Fix $d\ge 2$. The threshold $\lambda^*_{\spRade(\rho),d}$ from Theorem~\ref{thm:sparse-rad-lower} behaves as
$$\lambda^*_{\spRade(\rho),d} \ge 2\sqrt{-\rho \log \rho - O(\rho)}$$
in the limit $\rho \to 0$.
\end{proposition}

\noindent We defer the proof to Appendix~\ref{app:sparse-rad-p0}. When combined with the upper bound of Proposition~\ref{prop:cardinality-upper-bound}, this completes the proof of Theorem~\ref{thm:p0-sparse-rademacher} ($\rho \to 0$ limit).

\section{Proof of Theorem~\ref{thm:noise-conditioning}}
\label{sec:noise-conditioning-proof}

\subsection{Overview}

Let $P_n = \WT_n(d,\lambda,\cX)$ be the spiked model $T = \lambda x^{\otimes d} + W$, and let $Q_n = \WT_n(d)$ be the corresponding unspiked model $T = W$. Recall that $x$ is drawn from a prior $\cX$ with $\|x\| = 1$. Fix a sequence $\delta = \delta(n)$ such that $\frac{1}{\sqrt n} \ll \delta \ll 1$; for concreteness, say $\delta = \frac{\log n}{\sqrt n}$. Define the `good' event $\Omega(x,T)$ by $|\langle T, x^{\otimes d} \rangle - \lambda| \le \delta$. Note that under $P_n$, $\langle T,x^{\otimes d} \rangle \sim \lambda + \mathcal{N}(0,2/n)$ and so $\Omega(x,T)$ occurs with probability $1-o(1)$ as $n \to \infty$.

Let $\tilde P_n$ be the conditional distribution of $P_n$ conditioned on $\Omega(x,T)$. As discussed in Section~\ref{sec:contig}, the idea of our proof is to compute the $\chi^2$-divergence and show that it is small ($o(1)$ for $d \ge 3$, $O(1)$ for $d = 2$ as $n \to \infty$); dropping the additive $-1$ for now, our goal is to compute the second moment $M \defeq \EE_{Q_n}\left[ \left(\dd[\tilde P_n]{Q_n}\right)^2 \right]$. We have

\begin{align*}
\dd[\tilde P_n]{Q_n}(T)
&= \frac{(1+o(1))\, \EE_{x \sim \cX} \one_{\Omega(x,T)} \exp(-\frac{n}{4} \langle T - \lambda x^{\otimes d},T - \lambda x^{\otimes d} \rangle)}{\exp(-\frac{n}{4} \langle T,T \rangle)} \\
&= (1+o(1))\,\EE_{x \sim \cX} \one_{\Omega(x,T)} \exp\left(\frac{n}{2} \langle T, \lambda x^{\otimes d} \rangle - \frac{n\lambda^2}{4} \right)
\end{align*}
and so
$$M \defeq \Ex_{Q_n} \left(\dd[\tilde P_n]{Q_n}\right)^2 = (1+o(1)) \Ex_{x,x' \sim \cX} m(x,x')$$
where $x,x'$ are drawn independently from $\cX$ and
\begin{equation}
\label{eq:defn-m}
m(x,x') = \Ex_{T \sim Q_n} \one_{\Omega(x,T)}\one_{\Omega(x',T)} \exp\left(\frac{n\lambda}{2} \langle T, x^{\otimes d} + x'^{\otimes d} \rangle - \frac{n\lambda^2}{2} \right).
\end{equation}
We will see that $m(x,x')$ only depends on $\beta \defeq \langle x,x' \rangle^d$ and so we will write $m(\beta) \defeq m(x,x')$. For technical reasons we will break down the computation of $M = (1+o(1))\,\EE_\beta \,m(\beta)$ into two parts depending on the value of $|\beta|$:
\begin{itemize}
\item $[\eps,1]$ interval: $M_1 = \EE_\beta[\one_{|\beta| \in [\eps,1]} \,m(\beta)]$
\item $[0,\eps)$ interval: $M_2 = \EE_\beta[\one_{|\beta| \in [0,\eps)} \,m(\beta)]$
\end{itemize}
where $\eps > 0$ is a small constant (depending on $\lambda$ but not $n$) to be chosen later. Note that $\beta \in [0,1]$ for $d$ even and $\beta \in [-1,1]$ for $d$ odd. In the subsequent subsections we bound $M_1$ and $M_2$ separately. We begin with $M_1$, which is the crux of the proof.

\subsection{Interval $[\eps,1]$}

For $x,x'$ fixed and $T \sim Q_n$, the joint distribution of $(y,y') \defeq (\langle T, x^{\otimes d} \rangle, \langle T, x'^{\otimes d} \rangle)$ is the bivariate Gaussian $\mathcal{N}(0,\Sigma)$ where
$$\Sigma = \frac{2}{n} \left(\begin{array}{cc} 1 & \beta \\ \beta & 1 \end{array}\right)$$
where, recall, $\beta = \langle x,x' \rangle^d$. This distribution has density $\frac{1}{2 \pi} |\Sigma|^{-1/2} \exp\left(-\frac{1}{2} y^\top \Sigma^{-1} y\right)$ where
$$|\Sigma| = \frac{4}{n^2}(1-\beta^2)$$
and
$$\Sigma^{-1} = \frac{n}{2} (1-\beta^2)^{-1} \left(\begin{array}{cc} 1 & -\beta \\ -\beta & 1 \end{array}\right).$$

\noindent This means we have
\begin{align*}
m(\beta) &= \int_{y} \int_{y'} \exp\left(\frac{n\lambda}{2} (y + y') - \frac{n\lambda^2}{2} \right) \frac{n}{4 \pi}(1-\beta^2)^{-1/2} \exp\left(-\frac{n}{4(1-\beta^2)}(y^2 + y'^2 - 2\beta yy')\right) \dee y' \dee y
\intertext{where the integral is over the square of `good' $y,y'$ values: $|y - \lambda| \le \delta$, $|y' - \lambda| \le \delta$. Currently there are problematic terms that blow up when $|\beta| = 1$. The following change of variables is a trick to eliminate this issue; an alternative approach would be to treat the $|\beta| \in [1-\eps,\eps]$ interval separately, bounding $m(\beta)$ by $m(1)$. Letting $u = \frac{1}{2}(y+y')$ and $v = \frac{1}{2}(y-y')$ yields}
&= \frac{n}{2\pi} (1-\beta^2)^{-1/2} \int_u \int_v \exp\left[n \lambda u - \frac{n\lambda^2}{2} - \frac{n}{2} \left(\frac{u^2}{1+\beta} + \frac{v^2}{1-\beta}\right)\right] \dee v\, \dee u.
\intertext{Relaxing the integration bounds for $v$ to all of $\mathbb{R}$, explicitly compute the integral over $v$:}
&\le \sqrt{\frac{n}{2\pi (1+\beta)}} \int_u \exp\left(n \lambda u - \frac{n\lambda^2}{2} - \frac{nu^2}{2(1+\beta)} \right) \dee u.
\intertext{Since $|u - \lambda| \le \delta$ we can write $u = \lambda + \Delta$ with $|\Delta| \le \delta$:}
&= \sqrt{\frac{n}{2\pi (1+\beta)}} \int_{|\Delta| \le \delta} \exp\left(\frac{n}{2}\cdot \frac{\beta}{1+\beta}\, (\lambda^2 + 2 \lambda \Delta) - \frac{n \Delta^2}{2(1+\beta)}\right) \dee u.
\intertext{Upper bound this expression by removing the last term inside the exponent (which is always negative) and bounding $\Delta$ by $\delta$:}
&\le \sqrt{\frac{n}{2\pi (1+\beta)}} \exp\left(\frac{n}{2}\cdot \frac{\beta}{1+\beta}\, (\lambda^2 \pm 2 \lambda \delta)\right) \dee u
\end{align*}
\noindent where $\pm$ has the same sign as $\beta$. By computing the derivative with respect to $\beta$ we see that for $|\beta| \ge \eps$ and for sufficiently large $n$, the above bound is increasing in $\beta$. This means that up to a factor of 2 we only need to consider positive $\beta$, in which case $1/\sqrt{1+\beta} \le 1$. Letting $\tilde \lambda = \sqrt{\lambda^2 + 2\lambda \delta} = \lambda + o(1)$, we now have
$$m(\beta) \lesssim m_1(\beta) \defeq \sqrt{n} \exp\left(\frac{n \tilde\lambda^2}{2} \frac{\beta}{1+\beta}\right).$$

Using $m_1$ we will now bound $M_1$, the contribution of the $|\beta| \in [\eps,1]$ values to the second moment $M$. We have
\begin{align*}
M_1 &\lesssim \EE_\beta[\one_{\beta \in [\eps,1]} \, m_1(\beta)] \\
&= \int_{r=0}^\infty \problr{\one_{\beta \in [\eps,1]} \, m_1(\beta) \ge r} \dee r \\
&= \int_{r=0}^\infty \problr{\beta \in [\eps,1] \;\text{ and }\; \sqrt{n} \exp\left(\frac{n \tilde\lambda^2}{2} \frac{\beta}{1+\beta}\right) \ge r} \dee r. \\
\intertext{The change of variables $r = n \exp\left(\frac{n \tilde\lambda^2}{2} \frac{t}{1+t}\right)$ yields}
&= \int_{t=-1}^\infty \problr{\beta \in [\eps,1] \;\text{ and }\; \beta \ge t} \exp\left(\frac{n \tilde\lambda^2}{2} \frac{t}{1+t} \right) \frac{n^{3/2} \tilde\lambda^2}{2(1+t)^2} \,\dee t \\
&\le \int_{t=-1}^\eps \problr{\beta \ge \eps} \exp\left(\frac{n \tilde\lambda^2}{2} \frac{t}{1+t} \right) \frac{n^{3/2} \tilde\lambda^2}{2(1+t)^2} \,\dee t \; + \\
&{}\qquad \int_{t=\eps}^{1} \problr{\beta \ge t} \exp\left(\frac{n \tilde\lambda^2}{2} \frac{t}{1+t} \right) \frac{n^{3/2} \tilde\lambda^2}{2(1+t)^2} \,\dee t.
\intertext{Recalling the definition of rate function (Definition~\ref{def:rate-function}) and explicitly evaluating the integral in the first term,}
&\lesssim \sqrt{n} \exp\left(-n f_{n,\cX}(\eps^{1/d}) + \frac{n \tilde\lambda^2}{2} \frac{\eps}{1+\eps} \right) \; + \\
&{}\qquad n^{3/2} \int_{t=\eps}^1 \exp\left(-n f_{n,\cX}(t^{1/d}) + \frac{n \tilde\lambda^2}{2} \frac{t}{1+t} \right) \dee t.
\end{align*}

Recall that by condition (iii) of the theorem hypothesis, we are guaranteed $\lambda < \lambda^*$ and $f_\cX(t) \ge \frac{(\lambda^*)^2}{2} \frac{t^d}{1+t^d}$ for all $t \in (0,1)$. Combining this with the uniform convergence $f_{n,x} \ge b_{n,\cX} \to f_\cX$ (see Definition~\ref{def:rate-function}) and the fact $\tilde\lambda = \lambda + o(1)$ yields the following: there exists a constant $\gamma > 0$ (depending on $\lambda,\eps$ but not on $n$) such that for sufficiently large $n$ and for all $t \in [\eps,1)$ we have $f_{n,\cX}(t^{1/d}) \ge \frac{\tilde\lambda^2}{2} \frac{t}{1+t} + \gamma$. It then follows that both terms above are bounded above by $n^{3/2} \exp(-n \gamma) = o(1)$ and so we have $M_1 = o(1)$.

Note that this last step cannot be extended to all $t \in [0,1)$ because $f_\cX(t) = \frac{\lambda^2}{2} \frac{t^d}{1+t^d} = 0$ at $t = 0$; for this reason the $[0,\eps)$ interval requires separate consideration.

\subsection{Interval $[0,\eps)$}

For this case, the argument above is not enough and we will need to use the local subgaussian condition to control the deviations near zero. However, we do not need to use noise conditioning here and will revert to the standard second moment as in prior work (e.g.\ \cite{mrz,bmvx,pwbm}).

Starting from the definition (\ref{eq:defn-m}) of $m$, we drop the conditioning on $\Omega$ and use the Gaussian moment-generating function to explicitly compute the expectation over $T \sim Q_n$:
\begin{align*}
m(\beta) &\le \Ex_{T \sim Q_n} \exp\left(\frac{n\lambda}{2} \langle T, x^{\otimes d} + x'^{\otimes d} \rangle - \frac{n\lambda^2}{2} \right) \\
&= \exp\left(\frac{1}{2}\cdot\frac{2}{n}\cdot\frac{n^2\lambda^2}{4} \langle  x^{\otimes d} + x'^{\otimes d}, x^{\otimes d} + x'^{\otimes d} \rangle - \frac{n\lambda^2}{2} \right) \\
&= \exp\left(\frac{n\lambda^2}{2} \langle x^{\otimes d}, x'^{\otimes d} \rangle \right) \\
&= \exp\left(\frac{n\lambda^2}{2} \beta \right) \defeq m_2(\beta). \\
\end{align*}

\noindent A helpful trick in performing this computation is to note that $\langle T, x^{\otimes d} \rangle = \langle T', x^{\otimes d} \rangle$ where $T'$ is the asymmetric precursor tensor with $\cN(0,2/n)$ entries. Note that $m_2$ is an increasing function of $\beta$. Letting $b = 1$ for $d$ even and $b = 2$ for $d$ odd, we have
\begin{align*}
M_2 &= \prob{\beta = 0} + b \, \EE_\beta[\one_{\beta \in (0,\eps)} \, m_2(\beta)] \\
&= \prob{\beta = 0} + b\int_{r=0}^\infty \problr{\one_{\beta \in (0,\eps)} \, m_2(\beta) \ge r} \dee r \\
&= \prob{\beta = 0} + b\int_{r=0}^\infty \problr{\beta \in (0,\eps) \;\text{ and }\; \exp\left(\frac{n \lambda^2}{2} \beta \right) \ge r} \dee r. \\
\intertext{The change of variables $r = \exp\left(\frac{n \lambda^2}{2} s\right)$ yields}
&= \prob{\beta = 0} + b\int_{s=-\infty}^\infty \problr{\beta \in (0,\eps) \;\text{ and }\; \beta \ge s} \exp\left(\frac{n \lambda^2}{2} s \right) \frac{n \lambda^2}{2} \,\dee s \\
&\le \prob{\beta = 0} + b\int_{s=-\infty}^0 \prob{\beta > 0} \exp\left(\frac{n \lambda^2}{2} s \right) \frac{n \lambda^2}{2} \,\dee s + b\int_{s=0}^{\eps} \problr{\beta \ge s} \exp\left(\frac{n \lambda^2}{2} s \right) \frac{n \lambda^2}{2} \,\dee s \\
&= 1 + \frac{n \lambda^2 b}{2} \int_{s=0}^{\eps} \prob{\langle x,x' \rangle \ge s^{1/d}} \exp\left(\frac{n \lambda^2}{2} s \right) \dee s \\
&\defeq 1 + M_2'.
\end{align*}
Here we have used the fact $\prob{\beta = 0} + b \prob{\beta > 0} = 1$ (since $\langle x,x' \rangle$ has a symmetric distribution).

Letting $s = t^d$,
\begin{align}
\nonumber M_2' &\lesssim n \int_{t=0}^{\eps^{1/d}} \prob{\langle x,x' \rangle \ge t} \exp\left(\frac{n\lambda^2}{2} t^d\right) t^{d-1} \dee t.
\intertext{Let $\eta > 0$, to be chosen later. Choose $\eps$ small enough so that we can apply the local subgaussian bound, obtaining}
&\lesssim n \int_{t=0}^{\eps^{1/d}} \exp\left(-\frac{n}{2\sigma^2+\eta}t^2 + \frac{n\lambda^2}{2} t^d\right) t^{d-1} \dee t.
\label{eq:t2td}
\end{align}

First consider the $d \ge 3$ case. Here the $t^2$ term dominates the $t^d$ term, and so (by choosing $\eps$ sufficiently small) we have, for some constant $c > 0$,
$$M_2' \lesssim n \int_{t=0}^{\eps^{1/d}} \exp\left(-nc t^2\right) t^{d-1} \dee t
\le n \int_{t=0}^\infty \exp\left(-nc t^2\right) t^{d-1} \dee t
= \frac{n}{2} (cn)^{-d/2}\, \Gamma(d/2) = o(1).$$

Now consider the $d=2$ case. Now the $t^2$ and $t^d$ terms in (\ref{eq:t2td}) are of the same order, so it is important that we have the correct subgaussian constant $\sigma^2$. By hypothesis (v) of the theorem, we are guaranteed that $\lambda < 1/\sigma$ and so by choosing $\eta$ small enough we again have have a constant $c > 0$ such that
$$M_2' \lesssim n \int_{t=0}^{\eps^{1/d}} \exp\left(-nc t^2\right) t \,\dee t
\le n \int_{t=0}^\infty \exp\left(-nc t^2\right) t \,\dee t
= \frac{1}{2c} = O(1).$$

\noindent In conclusion, we have shown $M_2 = 1 + o(1)$ for $d \ge 3$, and $M_2 = O(1)$ for $d = 2$.

\subsection{Proof of non-detection}

Combining the results from the previous subsections, we have that the $\chi^2$-divergence
$$\chi^2(\tilde P_n \dmid Q_n) \defeq M - 1 = (1+o(1)) (M_1 + M_2) - 1$$
is $o(1)$ for $d \ge 3$ and $O(1)$ for $d = 2$. As discussed in Section~\ref{sec:contig}, this has the following consequences. If $d \ge 3$ we have $\TV(\tilde P_n,Q_n) = o(1)$; since $\TV(P_n,\tilde P_n) = o(1)$ this implies $\TV(P_n,Q_n) = o(1)$, and so weak detection is impossible. For all $d \ge 2$ we have $\tilde P_n \contig Q_n$, implying $P_n \contig Q_n$, and so strong detection is impossible.

\subsection{Proof of non-recovery}

\label{sec:pf-non-recovery}

We will show that in the $d \ge 3$ case, weak recovery is impossible. In this case we have shown $\TV(P_n,Q_n) = o(1)$. Therefore, in order to show that weak recovery is impossible under $P_n$, it is sufficient to show that weak recovery is impossible under $Q_n$ (i.e.\ we need to show that weak recovery is impossible when given only knowledge of the prior $\cX$ and nothing else). Assume on the contrary that weak recovery is possible under $Q_n$; this means there is a deterministic unit vector $v$ such that $\Pr_{x \sim \cX} [|\langle v,x \rangle| \ge \eps] \ge \eps$ for some $\eps > 0$ not depending on $n$. We will show that this allows us to perform weak detection, contradicting the above. Given a sample $T$ from either $P_n$ or $Q_n$, consider the statistic $m = \langle v^{\otimes d},T \rangle$ where $v$ is defined above. When $T \sim Q_n$ we have $m \sim \cN(0,2/n)$, which is $o(1)$ with probability $1-o(1)$ as $n\to\infty$. If instead $T \sim P_n$ then $m = \lambda \langle v,x \rangle^d + \langle v^{\otimes d},W \rangle$. The second term is $o(1)$ with probability $1-o(1)$, but the first term exceeds $\lambda \eps^d = \Theta(1)$ in absolute value with probability $\eps$. We therefore have the following distinguisher which performs weak detection: if $|m| \ge \lambda \eps^d/2$, output `$P_n$'; otherwise guess randomly.

\section*{Acknowledgements}
The authors are indebted to Ankur Moitra for helpful discussions that motivated the project, for feedback on a draft of this paper, and for providing guidance throughout. We thank G\'erard Ben Arous, Yash Deshpande, and Jonathan Weed for helpful conversations. We thank Yash Deshpande and Thibault Lesieur for resolving an issue in the replica predictions section of the first version of this paper.

\bibliographystyle{alpha}
\bibliography{bib}

\appendix

\section{Asymptotics for the spherical prior}\label{app:asymptotics}

Here we will derive the asymptotic expressions given in Section~\ref{sec:summary} for various bounds for the spherical prior.

\subsection{Lower bound $\lambda^*_{\sph,d}$}
Consider the threshold $\lambda \defeq \lambda^*_{\sph,d}$ from Theorem~\ref{thm:sph-lower}, i.e.\
$\lambda$ takes on the supremal value for which
$$\frac{\lambda^2}{2} \frac{t^d}{1+t^d} \leq -\frac12 \log(1-t^2)$$
for all $t \in [0,1)$. This inequality is not tight as $t \to 0$, considering the derivatives at $0$, nor as $t \to 1$, where the right-hand side becomes infinite while the left remains bounded. Hence there exists some point $t \in (0,1)$ where $\lambda^2 \frac{t^d}{1+t^d}$ and $-\log(1-t^2)$ agree in value and derivative:
$$ \lambda^2 \frac{t^d}{1+t^d} = - \log(1-t^2), \qquad d \lambda^2 t^{d-1} (1+t^d)^{-2} = \frac{t}{1-t^2}. $$
The latter equation yields $\lambda^2 = \frac{2 (1+t^d)^2}{d(1-t^2)t^{d-2}}$. Substituting into the first equation, we have
\begin{equation} 0 = \frac{2}{d} (1+t^d) \frac{t^2}{1-t^2} + \log(1-t^2). \label{eq:sph-problem}\end{equation}
This equation has a unique solution on $(0,1)$: if we divide by $\log(1-t^2)$, we have monotonicity as the derivative is negative:
\begin{align*}
\frac{2t(2t^2(1+t^d)+(2+t^d(2+d(1-t^2)))\log(1-t^2))}{d(1-t^2)^2 \log^2(1-t^2)} &\leq \frac{2t(2t^2(1+t^d)+(2+t^d(2+d(1-t^2)))(-t^2))}{d(1-t^2)^2 \log^2(1-t^2)} \\
&= \frac{-t^{2+d}}{(1-t^2)\log^2(1-t^2)} < 0.
\end{align*}

\noindent Suppose that we evaluate (\ref{eq:sph-problem}) at the value
$$ t = 1 - \frac{1}{c_1 d \log d  + c_2 d \log \log d + c_3 d + \eps(d)}, $$
with $\eps(d) = o(d)$ as $d \to \infty$. We note the expansions
$$ 1-t^2 = \frac{2}{c_1 d \log d + c_2 d \log \log d + c_3 d + \eps(d)} + O\left(\frac{1}{d^2 \log^2 d}\right), \quad 1+t^d = 2 - \frac{1}{c_1 \log d} + o\left(\frac{1}{\log d}\right). $$
Expanding (\ref{eq:sph-problem}), we obtain the equation
\begin{align*}
0 &= c_1 \log d - \frac12 \log d + c_2 \log \log d - \frac12 \log \log d + c_3 - \frac12 + \frac12 \log 2 - \frac12 \log c_1 + \eps(d) + O\left(\frac{\log \log d}{\log d}\right) \\
&= \eps(d) + O\left(\frac{\log \log d}{\log d}\right) \qquad \text{if $c_1 = c_2 = \frac12$, $c_3 = \frac12 - \log 2$.}
\end{align*}
With this choice of coefficients, the right-hand side is positive for a suitably large choice of $\eps(d) = o(d)$, and negative for a suitably small choice. By the intermediate value theorem, it follows that the unique solution of this equation in $(0,1)$ satisfies the asymptotics
$$ t = 1 - \frac{2}{d \log d + d \log \log d + 1 - 2 \log 2 + o(1)}, $$
which in turn implies:
\begin{align*}
    \lambda^2 &= \frac{2(1+t^d)^2}{d(1-t^2)t^{d-2}} \\
    &= \left( \frac12 \log d + \frac12 \log \log d + \frac12 + o(1) \right)\left(2 - \frac{2}{\log d} + o(1/\log d)\right)^2\left(1+\frac{2}{\log d}+o(1/\log d)\right) \\
    &= 2 \log d + 2 \log \log d + 2 - 4 \log 2 + o(1).
\end{align*}

\subsection{Injective norm $\mu_d$}

Recall from Theorem~\ref{thm:inj-norm} that the injective norm $\mu_d$ of $\WT(d)$ is given by $\mu_d = x\sqrt{2/d}$, where $x \geq 2\sqrt{d-1}$ is the unique solution to
$$ 0 = g(z) \defeq \frac{2-d}{d} - \log\left(\frac{dz^2}{2}\right) + \frac{d-1}{2} z^2 - \frac{2}{d^2 z^2}, \qquad z(x) = \frac{1}{(d-1)\sqrt{2d}} \left(x - \sqrt{x^2 - 4(d-1)}\right). $$
We will locate this root asymptotically. Let $x = \sqrt{d \log d + d \log \log d + d c}$, with $c = O(1)$. It is easily computed that
$$ z^2 = \frac{d^2}{2 (\log d + \log \log d + c - 2 + o(1))}, $$
so that $g(x) = 1 - c + o(1)$. Hence the unique root is located with $c = 1 + o(1)$, so that
$$ \mu_d^2 = 2 \log d + 2 \log \log d + 2 + o(1). $$

\subsection{Upper bound $\Lambda^*_{\sph,d}$}

We now develop an asymptotic understanding of the upper bound $\Lambda^*_{\sph,d}$ provided in Theorem~\ref{thm:sph-upper} and Corollary~\ref{cor:sph-upper}. Recall:

\begin{reptheorem}{thm:sph-upper}
Fix $d \ge 3$ and $\lambda \ge 0$. Let $\cX$ be any prior supported on the unit sphere in $\RR^d$. Let $T = \lambda x^{\otimes d} + W$ be a spiked tensor drawn from $\WT(d,\lambda,\cX)$. With probability $1-o(1)$ we have
$$\|T\| \ge L_d(\lambda) - o(1)$$
where
\begin{equation} L_d(\lambda) = \max_{m \in [0,1]} m^d \left(\lambda + \sqrt{\frac{2d}{d-1}} \sqrt{M (1+M)}\right), \qquad M(m) = (d-1) \frac{1-m^2}{m^2}.\end{equation}
\end{reptheorem}

\begin{repcorollary}{cor:sph-upper}
Fix $d \ge 3$ and $\lambda \ge 0$. Let $\cX$ be any prior supported on the unit sphere in $\RR^d$. Let $\mu_d$ be given by Theorem~\ref{thm:inj-norm} and let $L_d(\lambda)$ be defined as in Theorem~\ref{thm:sph-upper}. If $L_d(\lambda) > \mu_d$ then strong detection and weak recovery are possible for $\WT(d,\lambda,\cX)$.
\end{repcorollary}

Our main result will be the following:
\begin{claim}
Assuming $\lambda = \sqrt{2 \log d} + o(\sqrt{\log d})$, we have $L_d(\lambda)^2 - \lambda^2 \geq 2 + o(1)$ as $d \to \infty$.
\end{claim}
\noindent Recalling from the discussion above that $\mu_d^2 = 2 \log d + 2 \log \log d + 2 + o(1)$, it follows that $L_d(\lambda)$ exceeds $\mu_d$ as soon as $\lambda > 2 \log d + 2 \log \log d + o(1)$; at this point, reliable detection is possible according to Corollary~\ref{cor:sph-upper}. This expansion matches the replica prediction of the next section. We prove the claim:
\begin{proof}
To produce such a bound, we can choose a value of $m$; thus take $m = \exp(-c/d\log d)$, with $c$ a constant to be chosen later. Then $M = \frac{2}{c \log d} + O(d^{-1})$, and one computes that
\begin{align*}
    L_d(\lambda)^2 - \lambda^2 &\geq m^{2d} \left( \lambda^2 + 2 \sqrt{2} \lambda \sqrt{\frac{d}{d-1}} \sqrt{M(1+M)} + 2 \frac{d}{d-1} M(1+M) \right) - \lambda^2 \\
    &= \left( 1 - \frac{2}{c \log d}  + O(\log^{-2} d) \right) \left(\lambda^2 + 2 \sqrt{2} \lambda \sqrt{\frac{2}{c \log d}} + \frac{4}{c \log d} + o(1) \right) - \lambda^2 \\
    &= \frac{-2\lambda^2}{c \log d} + \frac{4\lambda}{\sqrt{c \log d}} + o(1) \\
    &= \frac{-4}{c} + 4 \sqrt{2/c},
\end{align*}
which is maximized at $c=2$ with a value of $2$.
\end{proof}
\noindent We conclude that the limiting (as $n \to \infty$, in probability) injective norm of a spiked tensor exceeds that of an unspiked tensor, thus guaranteeing reliable detection, as soon as $\lambda^2 > 2 \log d + 2 \log \log d + o(1)$.

\section{Replica predictions}\label{app:replica}

In this appendix we give a non-rigorous prediction for the exact statistical threshold for spiked tensor problems (as $n \to \infty$) with any fixed $d$, using the replica method from statistical physics (see \cite{mm-book} for an introduction). We expect the threshold predicted here to be correct for both strong detection and weak recovery, and when $d \ge 3$, also weak detection. Although the methods used here are non-rigorous, we have high confidence in their correctness because replica predictions have been rigorously shown to be correct in various related settings (e.g.\ \cite{tal06parisi,tal06,mi,mi-proof,lelarge-limits-lowrank}) using methods such as Guerra interpolation \cite{guerra} and the approximate message passing (AMP) framework \cite{amp-cs,amp-mot,bm,jm}.

We consider both the Rademacher prior and the spherical prior. (The sparse Rademacher prior can be approached using the same techniques; see also \cite{mmse-lowrank} for the solution to the $d=2$ case using different methods.) It will be more convenient to rescale the observations as
$$ T = \frac{\lambda d!}{2 n^{d-1}} x^{\otimes d} + W, $$
where $\|x\|^2 = n$ and the typical (distinct-index) entries of $W$ are distributed as $\cN(0,\frac{d!}{2n^{d-1}})$. With respect to the normalization in previous sections, this scaling preserves the meaning of the signal-to-noise ratio $\lambda$, while matching the noise scaling conventions of \cite{gardner} (on which the following computation will be based).

\subsection{Rademacher prior}

The replica calculation for this problem in the unspiked setting can be found in \cite{gardner}. This section adapts the computation to the spiked case. We observe the spiked tensor $T$ as above where each $x_i$ is uniformly $\pm 1$. Ignoring lower order terms (corresponding to non-distinct indices), the posterior distribution of $x$ given $T$ is
$$\prob{x \,|\, T} \propto \prod_{i_1 < \cdots < i_d} \exp\left(-\frac{n^{d-1}}{d!}(T_{i_1 \cdots i_d} - \frac{\lambda d!}{2n^{d-1}} x_{i_1} \cdots x_{i_d})^2\right)
\propto \exp\left(\lambda \sum_{i_1 < \cdots < i_d} T_{i_1 \cdots i_d} x_{i_1} \cdots x_{i_d} \right)$$
and so we are interested in the Boltzmann distribution over $\sigma \in \{\pm 1\}^n$ given by $\prob{\sigma \,|\, T} \propto \exp(-\beta H(\sigma))$ with energy $H(\sigma) = -\sum_{i_1 < \cdots < i_d} T_{i_1 \cdots i_d} \sigma_{i_1} \cdots \sigma_{i_d}$ and inverse temperature $\beta = \lambda$.

The goal is to compute the free energy density, defined as $f = -\frac{1}{\beta n} \EE \log Z$ where $$Z = \sum_{\sigma \in \{\pm 1\}^n} \exp(-\beta H(\sigma)).$$
The idea of the replica method is to instead compute the moments $\EE[Z^r]$ of $Z$ for $r \in \mathbb{N}$ and perform the (non-rigorous) analytic continuation
\begin{equation}
\label{eq:replica-trick}
\EE[\log Z] = \lim_{r \to 0} \frac{1}{r} \log \EE[Z^r].
\end{equation}

The moment $\EE[Z^r]$ can be expanded in terms of $r$ `replicas' $\sigma^1,\ldots,\sigma^r$ with $\sigma^a \in \{\pm 1\}^n$:
$$\EE[Z^r] = \sum_{\{\sigma^a\}} \EE\exp\left(\beta \sum_{i_1 < \cdots < i_d} T_{i_1 \cdots i_d} \sum_{a=1}^r \sigma_{i_1}^a \cdots \sigma_{i_d}^a\right).$$
After applying the definition of $T$ and the Gaussian moment-generating function (to compute expectation over the noise $W$) we arrive at
$$\EE[Z^r] = \sum_{\{\sigma^a\}} \exp\left[n\left(\frac{\lambda^2}{2} \sum_a c_a^d + \frac{\lambda^2}{4} \sum_{a,b} q_{ab}^d \right)\right]$$
where $q_{ab} = \frac{1}{n} \sum_i \sigma_i^a \sigma_i^b$ is the correlation between replicas $a$ and $b$, and $c_a = \frac{1}{n} \sum_i \sigma_i^a x_i$ is the correlation between replica $a$ and the truth.

Without loss of generality we can assume the true spike is $x = \one$ (all-ones). Let $Q$ be the $(r+1)\times (r+1)$ matrix of overlaps ($q_{ab}$ and $c_a$), including $x$ as the zeroth replica. Note that $Q$ is the average of $n$ \iid matrices and so by the theory of large deviations (Cram\'er's Theorem in multiple dimensions), the number of configurations $\{\sigma^a\}$ corresponding to given overlap parameters $q_{ab},c_a$ is asymptotically
\begin{equation}
\label{eq:rad-entropy}
\inf_{\mu,\nu} \exp\left[n\left(-\sum_a \nu_a c_a - \frac{1}{2} \sum_{a \ne b} \mu_{ab} q_{ab} + \log \sum_{\sigma \in \{\pm 1\}^r}\exp\left(\sum_a \nu_a \sigma_a + \frac{1}{2}\sum_{a \ne b} \mu_{ab} \sigma_a \sigma_b\right)\right)\right].
\end{equation}

We now apply the saddle point method: in the large $n$ limit, the expression for $\EE[Z^r]$ should be dominated by a single value of the overlap parameters $q_{ab},c_a$. This yields
$$\frac{1}{n} \log \EE[Z^r] = -G(q_{ab}^*,c_a^*,\mu_{ab}^*,\nu_a^*)$$
where $(q_{ab}^*,c_a^*,\mu_{ab}^*,\nu_a^*)$ is a critical point of
\begin{align*}
G(q_{ab},c_a,\mu_{ab},\nu_a) = &-\frac{\lambda^2}{2} \sum_a c_a^d - \frac{\lambda^2}{4} \sum_{a,b}  q_{ab}^d \\
&+ \sum_a \nu_a c_a + \frac{1}{2} \sum_{a \ne b} \mu_{ab} q_{ab} - \log \sum_{\sigma \in \{\pm 1\}^r}\exp\left(\sum_a \nu_a \sigma_a + \frac{1}{2}\sum_{a \ne b} \mu_{ab} \sigma_a \sigma_b\right).
\end{align*}

We next assume that the dominant saddle point takes a particular form: the so-called replica symmetric ansatz. The validity of this assumption is justified by a phenomenon in statistical physics: there is no static replica symmetry breaking on the Nishimori line (see e.g.\ \cite{statmech-survey}). The replica symmetric ansatz is given by $q_{aa} = 1$, $c_a = c$, $\nu_a = \nu$, and for $a \ne b$, $q_{ab} = q$ and $\mu_{ab} = \mu$ for constants $q,c,\mu,\nu$. This yields
$$\lim_{r \to 0} \frac{1}{r} G(q,c,\mu,\nu) = -\frac{\lambda^2}{2} c^d - \frac{\lambda^2}{4} + \frac{\lambda^2}{4} q^d + \nu c - \frac{1}{2} \mu(q-1) - \Ex_{z \sim \cN(0,1)} \log(2 \cosh(\nu + \sqrt{\mu} z))$$
where the last term is derived by first using the Gaussian moment-generating function to obtain the expression
$$\lim_{r \to 0} \frac{1}{r} \log \EE_z \sum_\sigma \exp((\nu + \sqrt{\mu} z)\sum_a \sigma_a) = \lim_{r \to 0} \frac{1}{r} \log \EE_z (2 \cosh(\nu + \sqrt{\mu}z))^r$$
and then applying the replica trick (\ref{eq:replica-trick}).

We next find the critical points by setting the derivatives of $G$ (with respect to all four variables) to zero, which yields
$$\nu = \frac{\lambda^2}{2} d c^{d-1}, \qquad \mu = \frac{\lambda^2}{2} dq^{d-1}, \qquad c = \EE_z \tanh(\nu + \sqrt{\mu}z), \qquad q = \EE_z \tanh^2(\nu + \sqrt{\mu}z).$$

\begin{figure}[!ht]
    \centering
    \begin{subfigure}[t]{0.47\textwidth}
        \centering
        \includegraphics[width=\linewidth]{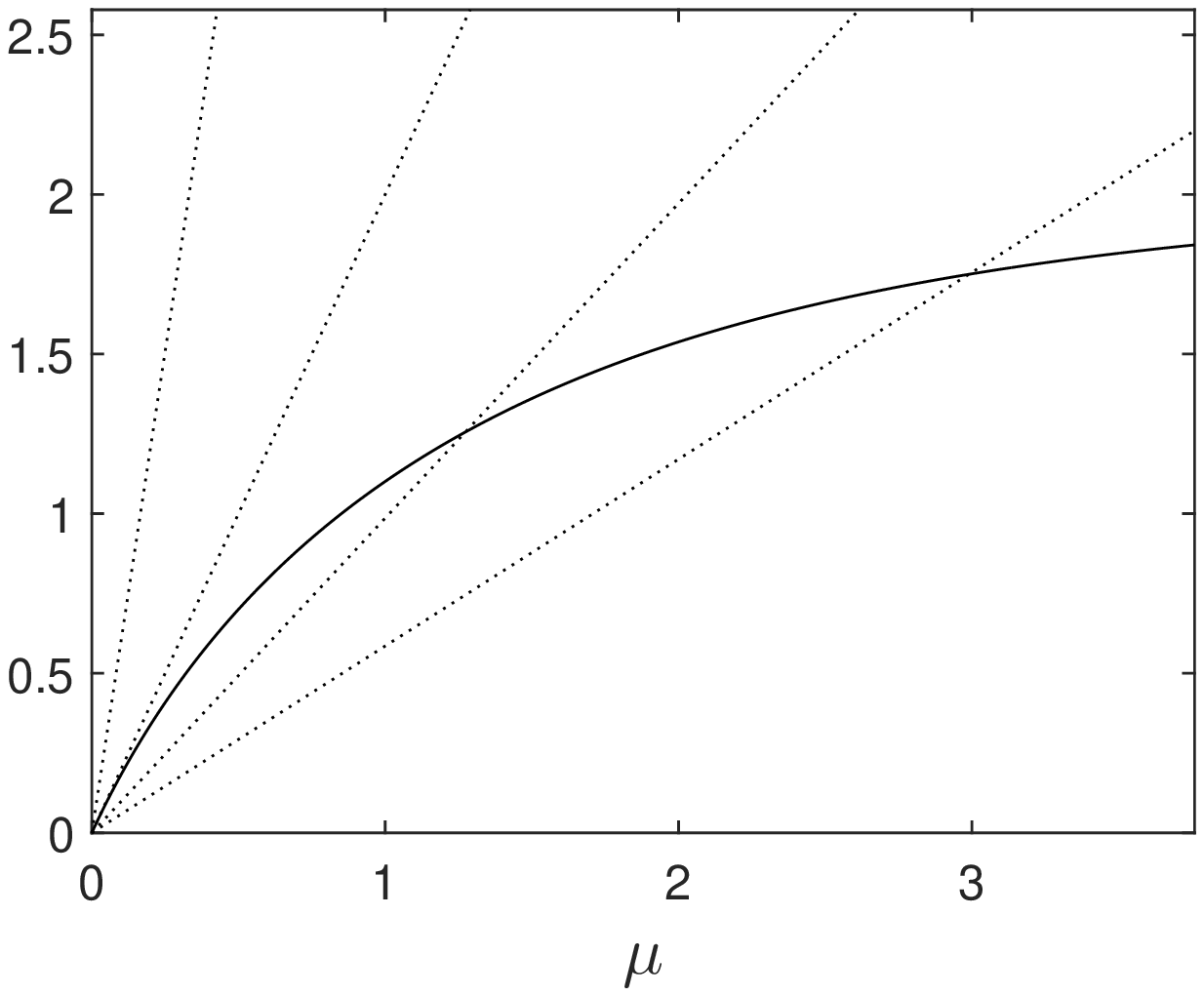}
    \end{subfigure}
    \hfill
    \begin{subfigure}[t]{0.47\textwidth}
        \centering
        \includegraphics[width=\linewidth]{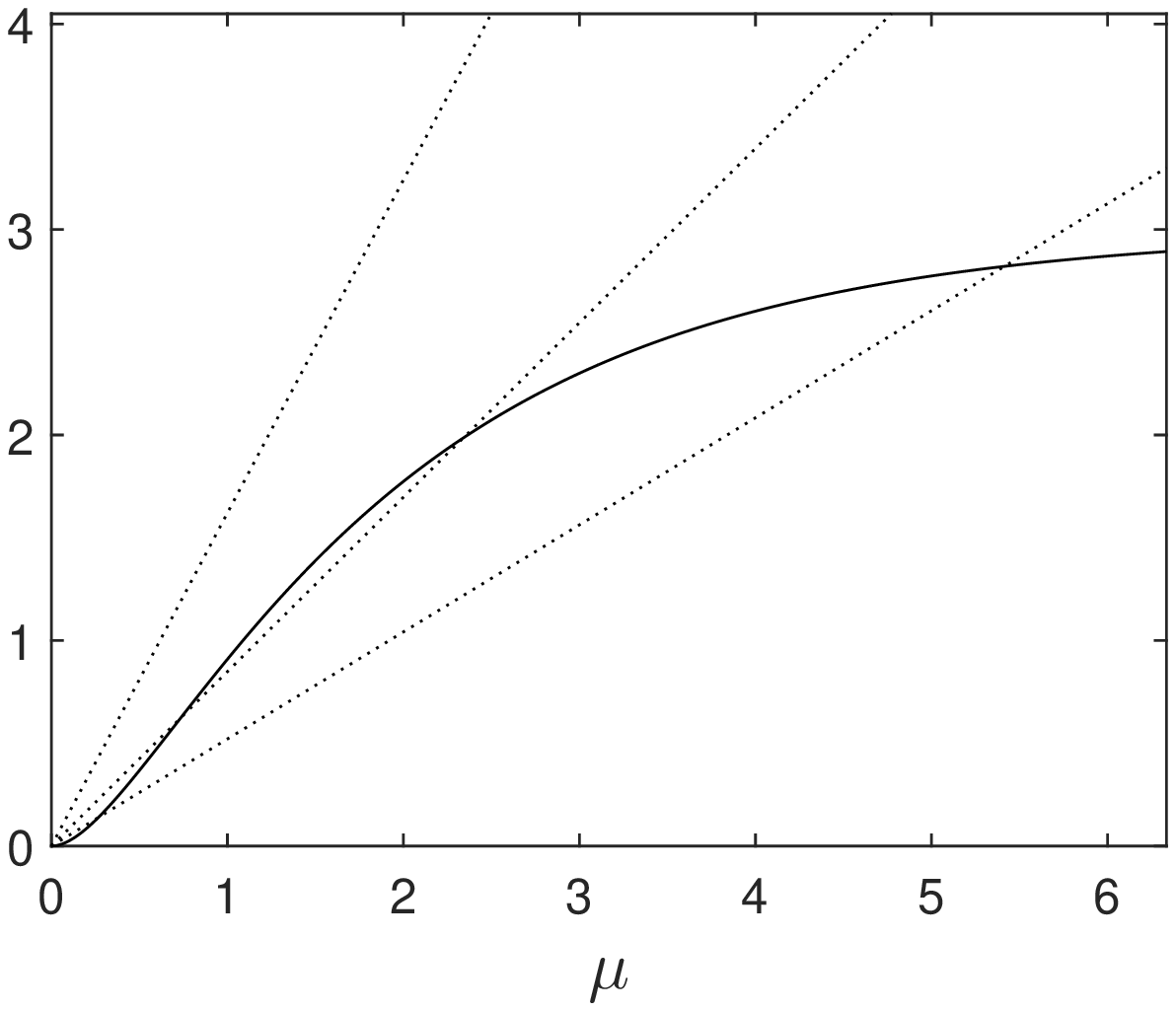}
    \end{subfigure}
    \caption{The replica symmetric solutions for the Rademacher prior, given by (\ref{eq:q-mu}). The left panel is $d = 2$ and the right panel is $d = 3$ (which is representative of all $d \ge 3$). The solid line is $d q^{d-1}$ as a function of $\mu$, where $q$ depends on $\mu$ via (\ref{eq:q-mu}). The dotted lines are $\frac{\mu}{\lambda^2}$ for various choices of $\lambda$ (with steeper lines corresponding to smaller $\lambda$). For a given $\lambda$, the intersections between the solid and dotted line are the solutions to (\ref{eq:q-mu}). Note that $\mu = 0$ is always a solution. For $d = 2$, a nonzero solution departs continuously from zero once $\lambda$ exceeds $1$. For $d \ge 3$, two nonzero solutions appear once $\lambda$ exceeds a particular value (which depends on $d$).}
    \label{fig:sol}
\end{figure}

\begin{figure}[!ht]
    \centering
    \begin{subfigure}[t]{0.47\textwidth}
        \centering
        \includegraphics[width=\linewidth]{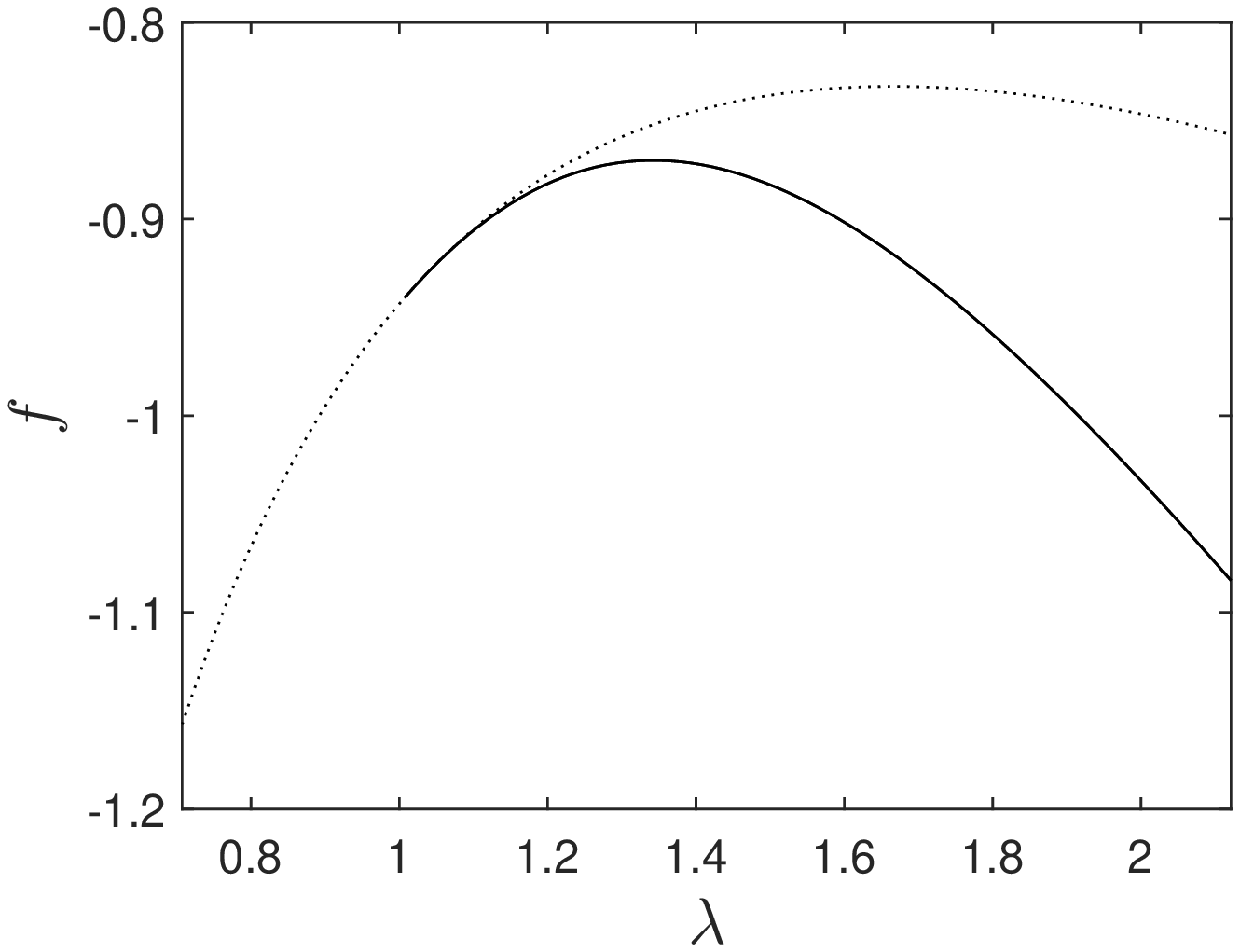}
    \end{subfigure}
    \hfill
    \begin{subfigure}[t]{0.47\textwidth}
        \centering
        \includegraphics[width=\linewidth]{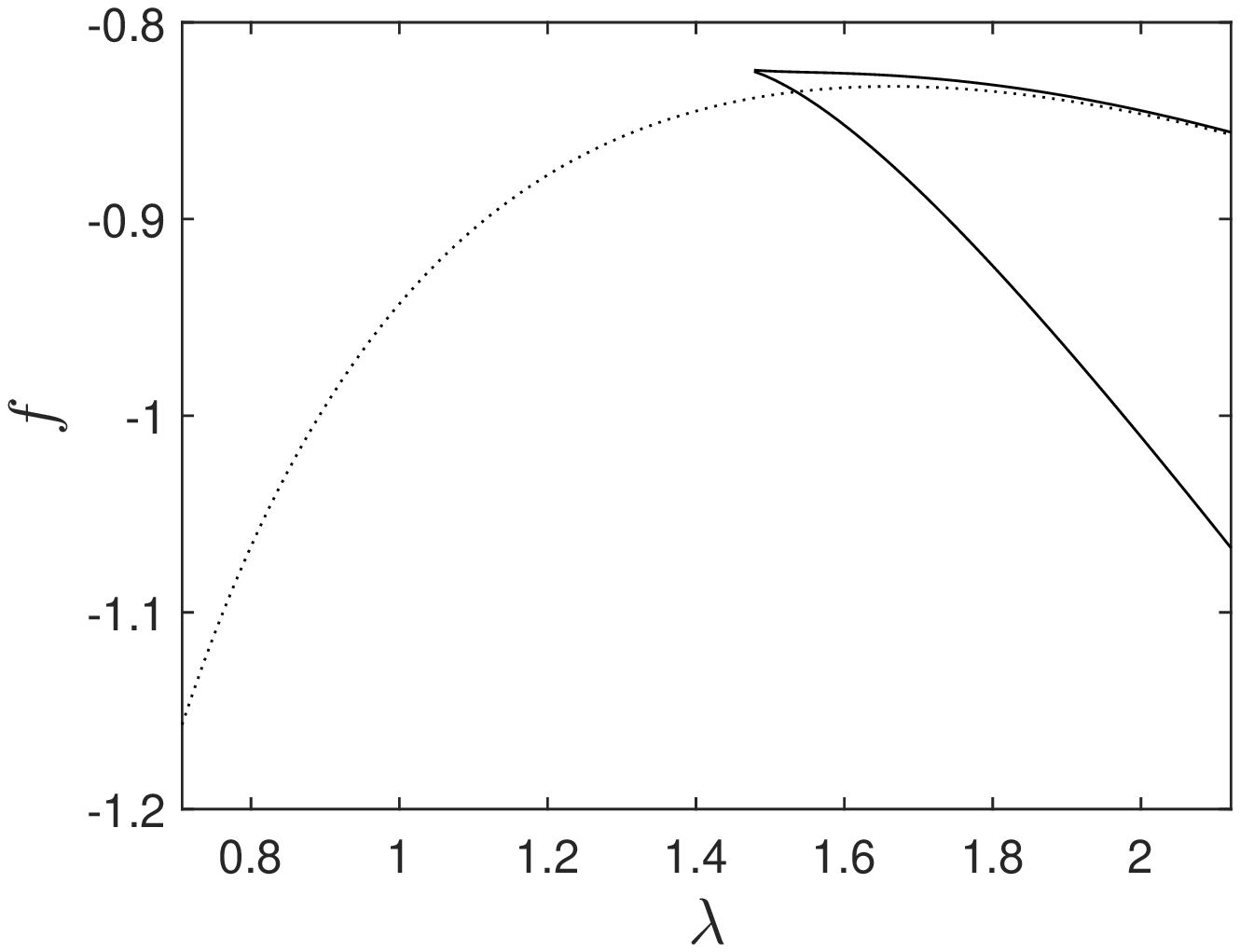}
    \end{subfigure}
    \caption{Free energy density of the replica symmetric solutions for the Rademacher prior (in the $n \to \infty$ limit), as a function of $\lambda$. The left panel is $d = 2$ and the right panel is $d = 3$ (which is representative of all $d \ge 3$). The dotted line is the zero solution $\mu = q = 0$. For $d = 2$, the solid line is the single nonzero solution. For $d = 3$, the solid lines are the two nonzero solutions with the bottom line corresponding to the larger (in terms of $\mu$) solution from Figure~\ref{fig:sol}. For each $\lambda$, the correct solution is the one of minimum free energy. The detection threshold is the value of $\lambda$ at which a nonzero solution first drops below the zero solution.}
    \label{fig:f}
\end{figure}

\begin{figure}[!ht]
    \centering
    \begin{subfigure}[t]{0.47\textwidth}
        \centering
        \includegraphics[width=\linewidth]{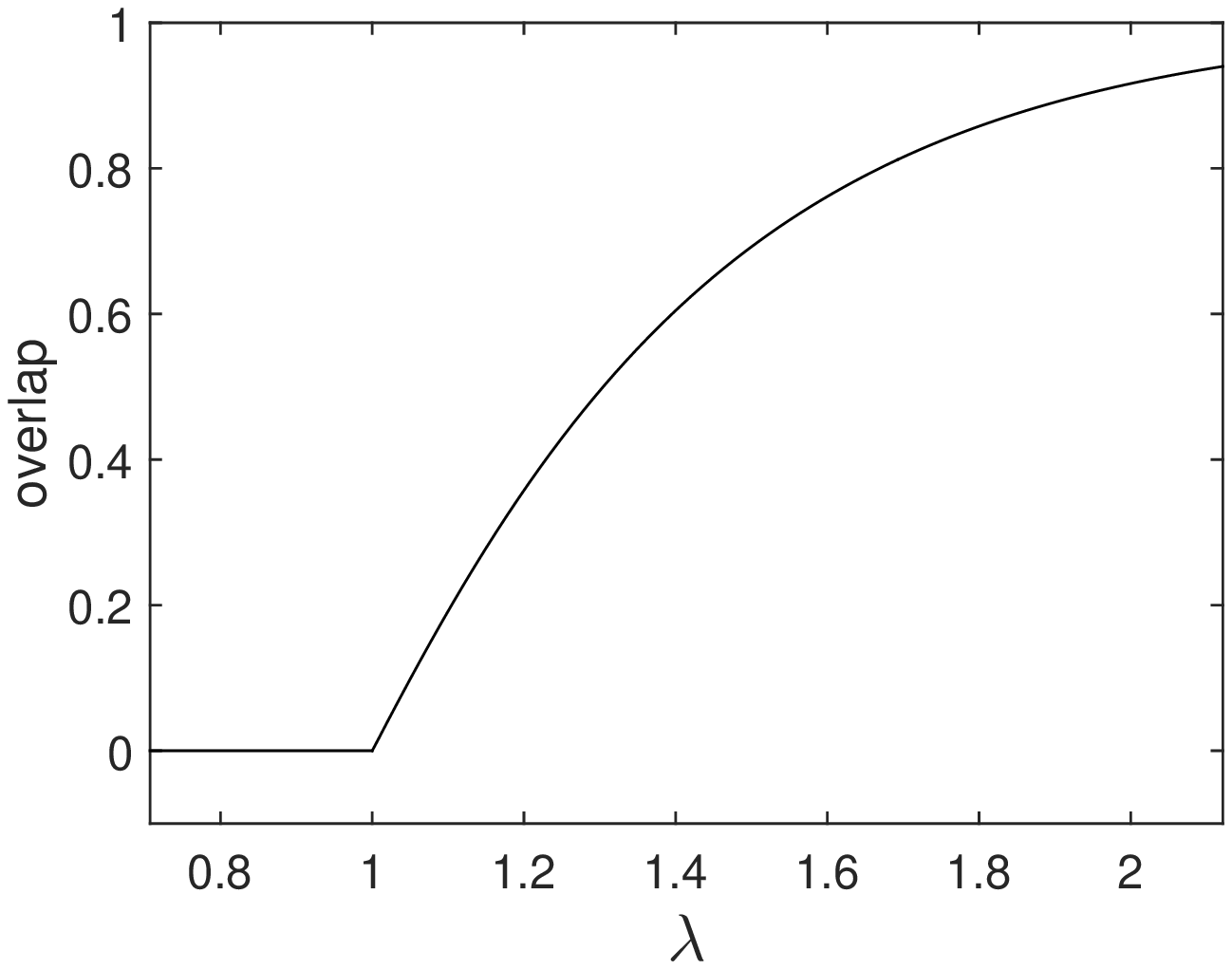}
    \end{subfigure}
    \hfill
    \begin{subfigure}[t]{0.47\textwidth}
        \centering
        \includegraphics[width=\linewidth]{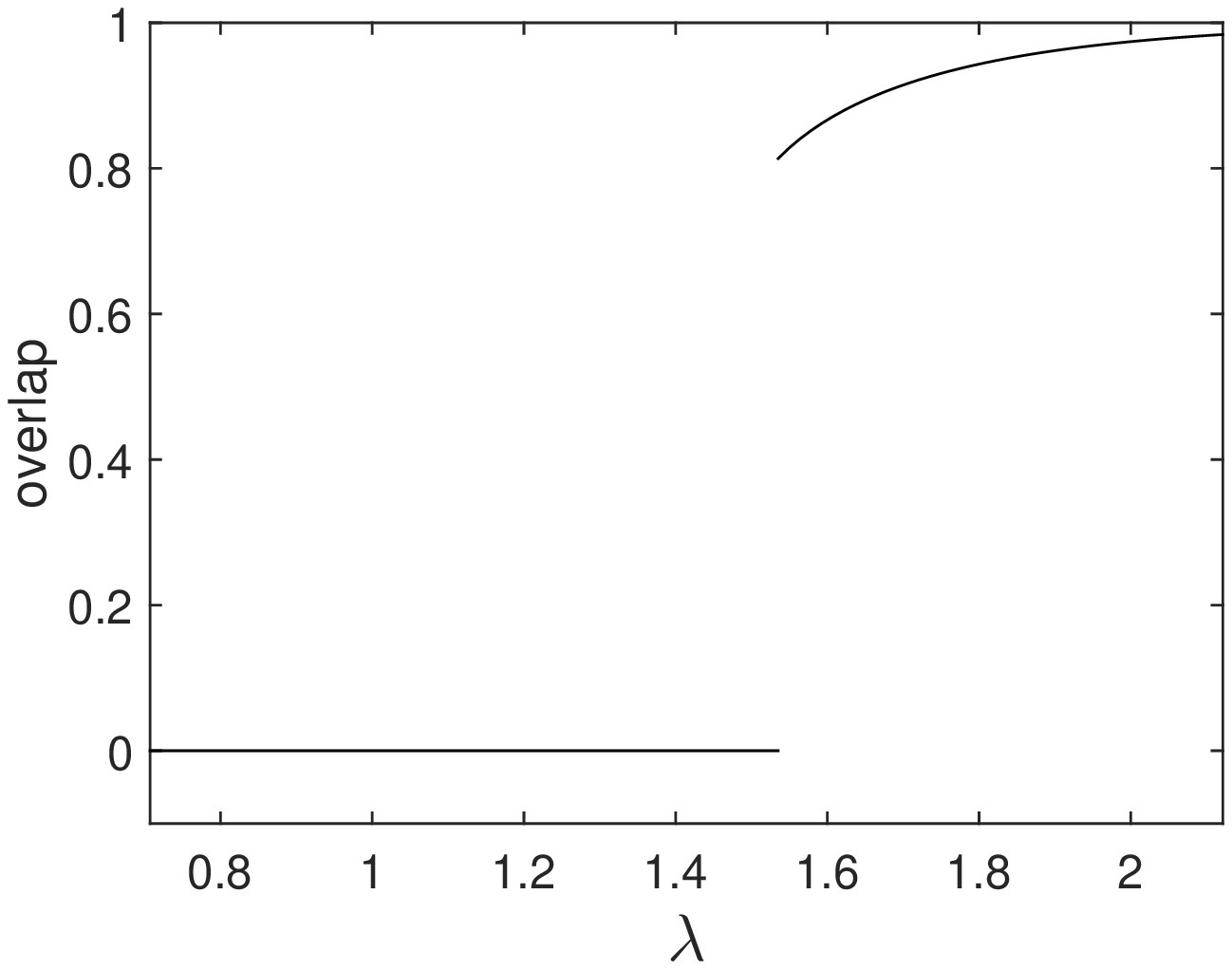}
    \end{subfigure}
    \caption{Overlap $q = \frac{1}{n} |\langle x,\hat x\rangle|$ between the true Rademacher spike $x$ and the MMSE estimator $\hat x$ (in the $n \to \infty$ limit), given by (\ref{eq:q-mu}). The left panel is $d = 2$ and the right panel is $d = 3$ (which is representative of all $d \ge 3$). The MMSE estimator has norm $\|\hat x\| = q \sqrt{n}$ or else its MSE could be improved by rescaling.}
    \label{fig:ov}
\end{figure}

Recall that the replicas are drawn from the posterior distribution $\prob{x \,|\, T}$ and so the truth $x$ behaves as if it is a replica; therefore we should have $c = q$. Using the identity $\EE_z \tanh(\gamma + \sqrt{\gamma}z) = \EE_z \tanh^2(\gamma + \sqrt{\gamma}z)$ (see e.g.\ \cite{dam}), we obtain the solution $c = q$ and $\nu = \mu$ where $q$ and $\mu$ are solutions to
\begin{equation}
\label{eq:q-mu}
\mu = \frac{\lambda^2}{2} d q^{d-1}, \qquad q = \Ex_{z \sim \cN(0,1)} \tanh(\mu + \sqrt{\mu}z).
\end{equation}
The free energy density of a solution to (\ref{eq:q-mu}) is given by
$$f = \frac{1}{\beta} \lim_{r \to 0} \frac{1}{r} G(q,c,\mu,\nu) = \frac{1}{\lambda}\left[-\frac{\lambda^2}{4} (q^d + 1) + \frac{1}{2} \mu (q+1) - \EE_z \log(2 \cosh(\mu + \sqrt{\mu}z))\right].$$

Figure~\ref{fig:sol} depicts the solutions to (\ref{eq:q-mu}). Note that (\ref{eq:q-mu}) always admits the zero solution $q = \mu = 0$. For $d = 2$, a nonzero solution appears once $\lambda$ exceeds $1$, recovering the eigenvalue transition of \cite{fp}. For $d \ge 3$, two nonzero solutions appear once $\lambda$ exceeds a particular value $\lambda_1^*$ (depending on $d$). The smaller of the two solutions is nonsensical because the overlap $q$ decreases with $\lambda$. In Figure~\ref{fig:f} we plot the free energy density of these solutions as a function of $\lambda$. For each $\lambda$ we expect that the correct solution is the one of minimum free energy density. Therefore we conclude that the detection threshold occurs at the point $\lambda_2^* > \lambda_1^*$ at which the free energy density of the larger nonzero solution drops below the free energy density of the zero solution. In Figure~\ref{fig:ov} we plot the overlap $q$ between the true spike and the optimal (MMSE) estimator; this is the same $q$ as in (\ref{eq:q-mu}). For $d = 2$ the overlap departs continuously from zero at the threshold, whereas for $d \ge 3$ it jumps discontinuously at the threshold.

For $d=2$, the result of this replica calculation has been rigorously shown to give the correct strong detection threshold ($\lambda = 1$) \cite{pwbm} as well as the correct recovery error for each $\lambda$ \cite{dam}.

For large $d$, the predicted detection threshold $\lambda_2^*$ approaches the constant $2 \sqrt{\log 2}$, matching the bounds in Theorem~\ref{thm:dinf-rademacher}. See Figure~\ref{fig:compare} for a comparison between the replica prediction and our rigorous upper and lower bounds.

For $d \ge 3$, it can be shown rigorously that the threshold for $\lambda$ given by the replica calculation is an upper bound on the true threshold, both for strong detection and weak recovery. For recovery, this follows from Theorem~1 of \cite{km} combined with the I-MMSE relation \cite{I-MMSE} along with the relation between free energy and mutual information (similarly to equation (10) in \cite{mi}). For detection, this follows again from Theorem~1 of \cite{km}, along with the fact (stated in \cite{km}, based on \cite{gt-limit}) that the free energy concentrates and thus provides reliable detection above the threshold. One can also deduce a rigorous lower bound on the weak recovery threshold using Theorem~3 of \cite{km} but it is weaker than our rigorous lower bound (and in fact, numerically appears to coincide with the lower bound of the basic second moment method without noise conditioning).

\FloatBarrier

\subsection{Spherical prior}

The computation in this section is similar to that of the previous section, so we will skip many of the details. This time we follow the corresponding unspiked computation in \cite{cs92}. The problem setup and normalization is the same as in the previous section except now the spike $x$ is drawn uniformly from vectors of norm $\sqrt{n}$. The only difference between this calculation and the Rademacher one is that we need to change the ``entropy term'' (\ref{eq:rad-entropy}). We have the free energy density
$$f = \frac{1}{\beta} \lim_{r \to 0} \frac{1}{r} G(q_{ab}^*,c_a^*)$$
where, borrowing the ``entropy term'' from equation (3.17) in \cite{cs92}, $(q_{ab}^*,c_a^*)$ is a critical point of
$$G(q_{ab},c_a) = -\frac{\lambda^2}{4}\sum_{a,b} q_{ab}^d - \frac{\lambda^2}{2} \sum_a c_a^d - \frac{1}{2}\log \det(Q)$$
where $Q$ is the $(r+1)\times (r+1)$ matrix of overlaps ($q_{ab}$ and $c_a$), including the true spike as the zeroth replica. Again, since we are on the Nishimori line, the replica symmetric ansatz should be valid: $q_{aa} = 1$, $q_{ab} = c_a = q$ for $a \ne b$. This yields
$$G(q) = -\frac{\lambda^2}{4} [r + r(r+1)q^d] - \frac{1}{2}\log[(1-q)^r((r+1)q+1-q)]$$
and so
$$\lim_{r \to 0} \frac{1}{r} G(q) = -\frac{\lambda^2}{4} (q^d + 1) - \frac{q}{2} - \frac{1}{2} \log(1-q).$$
Solving for stationary points in $q$,
\begin{equation}
\label{eq:sph-q}
\frac{\lambda^2}{2} d q^{d-1}(1-q) = q.
\end{equation}
The free energy density of a solution to (\ref{eq:sph-q}) is given by
$$f = \frac{1}{\beta} \lim_{r \to 0} \frac{1}{r} G(q) = \frac{1}{\lambda}\left[-\frac{\lambda^2}{4} (q^d + 1) - \frac{q}{2} - \frac{1}{2} \log(1-q)\right].$$

Qualitatively, the results are similar to the Rademacher prior of the previous section. For $d = 2$, a single nonzero solution to (\ref{eq:sph-q}) appears once $\lambda$ exceeds $1$. For $d \ge 3$, two solutions appear once $\lambda$ exceeds some value $\lambda_1^*$ (depending on $d$). The smaller of these solutions is nonsensical because $q$ decreases with $\lambda$. The larger solution becomes correct when $\lambda$ exceeds $\lambda_2^* > \lambda_1^*$, at which point its free energy drops below the free energy of the zero solution. The detection threshold occurs at $\lambda_2^*$.

For the matrix case ($d=2$), we recover the threshold $\lambda=1$ corresponding to the eigenvalue transition, as well as the correct overlap $q$ between the MMSE estimator and the truth (see Theorem~\ref{thm:bbp-wig}); the MMSE estimator is (asymptotically as $n\to\infty$) given by rescaling the top eigenvector to have norm $q \sqrt{n}$.

Following the argument style of Appendix~\ref{app:asymptotics}, one computes that as $d \to \infty$, the squared predicted detection threshold behaves as
$$ (\lambda_2^*)^2 = 2 \log d + 2 \log \log d + o(1), $$
implying the cruder statement $\lambda_2^* = \sqrt{2 \log d} + o(1)$. See Figure~\ref{fig:compare} for a comparison between the replica prediction and our rigorous upper and lower bounds.

\section{Sparse Rademacher tail bounds}
\label{app:sparse-rad-tail}

Throughout this section we fix $\rho \in (0,1]$ and consider the sparse Rademacher prior $\cX_\spRade(\rho)$. The notation $\simeq, \lesssim$ will denote statements that are true up to constants that may depend on $\rho$ but not on any other variable. Recall from Section~\ref{sec:sparse-rad} the sparse Rademacher rate function, which we will denote by $f$:
\begin{equation}
\label{eq:sparse-rate}
f(t) = \min_{\zeta \in [\max(\rho t,2\rho-1),\rho]} G(\zeta) + \zeta f_\Rade\left(\frac{\rho t}{\zeta}\right)
\end{equation}
where
$$G(\zeta) = -H(\{\zeta,p-\zeta,p-\zeta,1-2\rho+\zeta\}) + 2H(\rho)$$
and where $f_\Rade$ is the Rademacher rate function:
$$f_\Rade(t) = \log 2 - H\left(\frac{1+t}{2}\right).$$

In this section (as discussed in Section~\ref{sec:sparse-rad}) we prove the necessary tail bounds to show that $f$ indeed satisfies our definition of rate function.

\begin{repproposition}{prop:sparse-tail-weak}
Let $x,x'$ be drawn independently from the sparse Rademacher prior $\cX_\spRade(\rho)$. Then for any $t \in [0,1]$ we have
$$\prob{\langle x,x' \rangle \ge t} \lesssim n^{3/2} \exp(-n f(t))$$
where $f$ is the rate function \emph{(\ref{eq:sparse-rate})}.
\end{repproposition}

\begin{proof}
Let $z$ be the number of indices on which the supports of $x$ and $x'$ overlap, and let $\zeta = z/n$. Note that the set of possible $\zeta$ values is precisely the interval $[\max(0,2\rho-1),\rho]$. In order to have $\prob{\langle x,x' \rangle \ge t}$ we also need $\zeta \ge \rho t$, which restricts us to the interval $\mathcal{I} \defeq [\max(0,2\rho-1,\rho t),\rho]$ appearing in (\ref{eq:sparse-rate}). Using the Rademacher Chernoff bound we can write
\begin{equation}
\label{eq:z-sum}
\prob{\langle x,x' \rangle \ge t} \le \sum_z \prob{z} \exp\left(-n \zeta f_\Rade\left(\frac{\rho t}{\zeta}\right)\right).
\end{equation}
where $z$ ranges over all integer values for which $\zeta = z/n \in \mathcal{I}$. Note that $z$ follows the hypergeometric distribution $z \sim \mathrm{Hyp}(n,\rho n,\rho n)$ given by
$$\prob{z} = \frac{\binom{\rho n}{z} \binom{(1-\rho)n}{\rho n-z}}{\binom{n}{\rho n}}.$$
By Stirling's approximation we have $n! \simeq n^{n+\frac{1}{2}}e^{-n}$ and so $\binom{n}{k} \simeq \sqrt{\frac{n}{k(n-k)}} \exp(n H(k/n))$. (By $A \simeq B$ we mean there exist positive constants $C_1,C_2$ such that $C_1 A \le B \le C_2 A$.) This means
\begin{align*}
\prob{z} &\lesssim \sqrt{n} \exp\left\{n\left[\rho H\left(\frac{\zeta}{\rho}\right) + (1-\rho)H\left(\frac{\rho-\zeta}{1-\rho}\right) - H(\rho)\right]\right\} \\
&= \sqrt{n} \exp\left\{n\left[H(\{\zeta,p-\zeta,p-\zeta,1-2\rho+\zeta\}) - 2H(\rho)\right]\right\} \\
&= \sqrt{n} \exp\left(-n G(\zeta)\right).
\end{align*}
The result now follows by bounding the sum in (\ref{eq:z-sum}) by $n$ times its largest term.
\end{proof}

\begin{repproposition}{prop:sparse-tail-strong}
For every $\rho \in (0,1]$ there exist constants $C,T > 0$ such that when $x,x'$ are drawn independently from $\cX_\spRade(\rho)$, we have
$$\prob{\langle x,x' \rangle \ge t} \le C \exp(-n f(t)) \qquad \forall t \in [0,T].$$
\end{repproposition}

\begin{proof}
For small $t$, the last term $\zeta f_\Rade\left(\frac{\rho t}{\zeta}\right)$ in the definition of $f$ becomes negligible (converges to zero uniformly in $\zeta \in [\rho t,\rho]$ as $t \to 0$ by Dini's theorem) and so $f(t)$ becomes approximately the minimum value of $G(\zeta)$ over the valid range $\zeta \in [\max(\rho t,2\rho-1),\rho]$.
The minimum of $G(\zeta)$ occurs at $\zeta = \rho^2$
and furthermore $G(\zeta)$ is strongly convex. Let $\zeta^* = \zeta^*(t)$ be the minimizer in the expression for the rate function (\ref{eq:sparse-rate}) so that $f(t) = G(\zeta^*) + \zeta^* f_\Rade\left(\frac{\rho t}{\zeta^*}\right)$. The above implies $\zeta^* \to \rho^2$ as $t \to 0$. We also have $\frac{\partial^2}{\partial \zeta^2}\left[\zeta f_\Rade\left(\frac{\rho t}{\zeta}\right)\right] = \frac{\rho^2 t^2}{\zeta^3 - \rho^2 t^2 \zeta}$, which converges uniformly (in $\zeta$) to zero as $t \to 0$ in a neighborhood (for $\zeta$) containing $\rho^2$.

It follows from the above that we can choose $T$ small enough so that there exists $\delta > 0$ (depending on $\rho$ but not $t$) so that
$$G(\zeta) + \zeta f_\Rade\left(\frac{\rho t}{\zeta}\right) \ge f(t) + \delta (\zeta - \zeta^*)^2 \qquad \forall t \in [0,T], \zeta \in \mathcal{I}.$$

Let $\eps$ be a small constant (depending on $\rho$ but not $t$) to be chosen later. We first consider the sum (\ref{eq:z-sum}) over the values $\zeta \in \mathcal{I}' \defeq [\max(0,2\rho-1) + \eps, \rho - \eps]$, i.e. discarding a small interval at each endpoint of $\mathcal{I}$. For $\zeta \in \mathcal{I}'$, Stirling's approximation yields (see the proof of Proposition~\ref{prop:sparse-tail-weak} above)
$$\prob{z} \simeq \frac{1}{\sqrt n} \exp\left(-nG(\zeta)\right).$$

\noindent Now, starting from (\ref{eq:z-sum}), we have
\begin{align*}
\sum_{z \, :\, \zeta \in \mathcal{I}'} \prob{z} \exp\left(-n \zeta f_\Rade\left(\frac{\rho t}{\zeta}\right)\right)
&\lesssim \frac{1}{\sqrt n} \sum_{z \, :\, \zeta \in \mathcal{I}'} \exp\left\{-n\left[G(\zeta) + \zeta f_\Rade\left(\frac{\rho t}{\zeta}\right)\right]\right\} \\
&\le \frac{1}{\sqrt n} \sum_{z} \exp\left(-n f(t) - n \delta (\zeta - \zeta^*)^2 \right) \\
&= \exp(-n f(t)) \cdot \frac{1}{\sqrt n} \sum_{z} \exp\left(-n \delta (\zeta - \zeta^*)^2 \right) \\
&\lesssim \exp(-nf(t))
\end{align*}
where the last step follows from (the proof of) Lemma~10 in \cite{bmnn}.

It remains to bound the terms in (\ref{eq:z-sum}) for which $\zeta$ lies outside $\mathcal{I}'$. First consider $\zeta \in [0,\eps]$ (in the case $\rho < \frac{1}{2}$). Here, instead of applying Stirling's approximation to $\binom{\rho n}{z}$, we apply the simpler bound $\binom{n}{k} \le \left(\frac{ne}{k}\right)^k$ to yield $\binom{\rho n}{z} \le \binom{\rho n}{\eps n} \le \left(\frac{\rho e}{\eps}\right)^{\eps n}$. As above, we apply Stirling's approximation to the other two binomial coefficients in $\prob{z}$, yielding
$$\frac{\binom{(1-\rho)n}{\rho n-z}}{\binom{n}{\rho n}} \simeq \exp\left\{n\left[(1-\rho)H\left(\frac{\rho-\zeta}{1-\rho}\right) - H(\rho)\right]\right\} \le \exp(-nG(\zeta)).$$
Also, since $\zeta \in [0,\eps]$ is far from the minimizer $\zeta^* \approx \rho^2$ (for small $t$), there exists a constant $\Delta$ (depending on $\rho$) such that
$$G(\zeta) + \zeta f_\Rade\left(\frac{\rho t}{\zeta}\right) \ge f(t) + \Delta \qquad \forall t \in [0,T], \zeta \in [0,\eps].$$
This yields
\begin{align*}
\sum_{z \, :\, \zeta \in [0,\eps]} \prob{z} \exp\left(-n \zeta f_\Rade\left(\frac{\rho t}{\zeta}\right)\right)
&\lesssim \sum_{z \, :\, \zeta \in [0,\eps]} \left(\frac{\rho e}{\eps}\right)^{\eps n} \exp\left\{-n\left[G(\zeta) + \zeta f_\Rade\left(\frac{\rho t}{\zeta}\right)\right]\right\} \\
&\le \sum_{z \, :\, \zeta \in [0,\eps]} \left(\frac{\rho e}{\eps}\right)^{\eps n} \exp(-nf(t) - n \Delta) \\
&\le \exp(-nf(t)) \cdot \eps n \left[\left(\frac{\rho e}{\eps}\right)^\eps e^{-\Delta} \right]^n,
\end{align*}
which is $o(1) \exp(-nf(t))$ provided we choose $\eps$ small enough so that the term in brackets $[\cdots]$ is less than 1; this is possible because $\lim_{\eps \to 0^+} \left(\frac{\rho e}{\eps}\right)^\eps = 1$.

A similar argument applies to the other boundary cases: $\zeta \in [2\rho-1,2\rho-1 + \eps]$ (for $\rho > \frac{1}{2}$) and $\zeta \in [\rho - \eps, \rho]$.
\end{proof}

\section{Sparse Rademacher asymptotics}
\label{app:sparse-rad-p0}

In this section we analyze the $\rho \to 0$ asymptotics for our lower bound on the sparse Rademacher threshold.

\begin{repproposition}{prop:sparse-rad-p0}
Fix any $d \ge 2$. Let $f_\rho(t) = f_{\cX_\spRade (\rho)}$, the sparse Rademacher rate function. Let $\lambda^* = \lambda^*(\rho)$ be our lower bound on the detection threshold for the sparse Rademacher prior, namely
$$(\lambda^*)^2 = 2 \inf_{t \in (0,1)} \frac{1+t^d}{t^d} f_\rho(t).$$
In the limit $\rho \to 0$, we have
$$\lambda^* \ge 2\sqrt{-\rho \log \rho - O(\rho)}.$$
\end{repproposition}

\begin{proof}
Since $x \mapsto \frac{1+x}{x}$ is a decreasing function for $x > 0$, it is sufficient to prove the statement for $d = 2$. Recall the definition of $f_\rho$ (for $\rho < \frac{1}{3}$):
$$f_\rho(t) = \min_{\zeta \in [\rho t,\rho]} G(\zeta) + \zeta f_\Rade\left(\frac{\rho t}{\zeta}\right)$$
where
$$G(\zeta) = -H(\{\zeta,p-\zeta,p-\zeta,1-2\rho+\zeta\}) + 2H(\rho).$$
We can bound the Rademacher rate function $f_\Rade$ by using the bound $H(a) \le \log 2 - 2(a-\frac{1}{2})^2$ as follows:
$$f_\Rade(t) \defeq \log 2 - H\left(\frac{1+t}{2}\right) \ge \frac{t^2}{2}.$$

\noindent Our lower bound shows non-detection provided that
$$\lambda^2 < 2 \inf_{t \in (0,1)}\, \min_{\zeta \in [\rho t,\rho]} \frac{1+t^2}{t^2}\left(G(\zeta) + \frac{\rho^2 t^2}{2 \zeta}\right) = 2 \inf_{\zeta \in (0,\rho]} \, \inf_{t \in (0,\zeta/\rho]} (1+t^2)\left(\frac{G(\zeta)}{t^2} + \frac{\rho^2}{2\zeta}\right),$$
or equivalently,
\begin{equation}
\label{eq:t-opt}
\lambda^2 < 2 \inf_{t \in (0,\zeta/\rho]} (1+t^2)\left(\frac{G(\zeta)}{t^2} + \frac{\rho^2}{2\zeta}\right) \qquad \forall \zeta \in (0,\rho].
\end{equation}

\noindent We perform the optimization over $t$ in closed form. Since $G(\zeta) \ge 0$, the function $t \mapsto (1+t^2)\left(\frac{G(\zeta)}{t^2} + \frac{\rho^2}{2\zeta}\right)$ is convex on $t \in (0,\infty)$ with its minimum at $t^2 = \frac{1}{\rho}\sqrt{2\zeta G(\zeta)}$. We will break into two cases depending on whether this minimum occurs before or after the maximum allowable value $\zeta/\rho$ for $t$.

First consider the $\zeta$ values satisfying
\begin{equation}
\label{eq:zeta-cond}
\frac{1}{\rho}\sqrt{2\zeta G(\zeta)} \ge \frac{\zeta^2}{\rho^2}
\end{equation}
so that the infimum over $t$ in (\ref{eq:t-opt}) is optimized at $t = \zeta/\rho$. In this case we require
$$\lambda^2 < 2 (1 + \zeta^2/\rho^2)\left(\frac{G(\zeta) \rho^2}{\zeta^2}+\frac{\rho^2}{2\zeta}\right) = 2 \left(\frac{\rho^2 + \zeta^2}{\zeta^2}\right)\left(G(\zeta) + \frac{\zeta}{2}\right),$$
or equivalently,
$$G(\zeta) + \frac{\zeta}{2} > \frac{\lambda^2}{2} \cdot\frac{\zeta^2}{\rho^2+\zeta^2}.$$
If $\zeta < \frac{\rho^2}{\lambda^2}$ then this is satisfied because $G(\zeta) \ge 0$, so we can assume $\zeta \ge \frac{\rho^2}{\lambda^2}$. We will use the bound
\begin{equation}
\label{eq:G-bound}
G(\zeta) \ge \zeta \log\left(\frac{\zeta(1-\rho)}{ep^2}\right).
\end{equation}

\noindent The motivation for the bound (\ref{eq:G-bound}) comes from the argument in \cite{bmvx} where they compare the hypergeometric random variable to a binomial. One can prove (\ref{eq:G-bound}) directly by showing that $G(\zeta) - \zeta \log\left(\frac{\zeta(1-\rho)}{ep^2}\right)$ is jointly convex in $\zeta,\rho$ (since the Hessian has positive trace and determinant) and has a stationary point in the limit $\zeta \to 0, \rho \to 0$ (at which point the bound holds).

Combining (\ref{eq:G-bound}) with $\frac{\rho^2}{\lambda^2} \le \zeta \le \rho$ and the fact that $\zeta \mapsto \frac{\zeta}{\rho^2 + \zeta^2}$ is increasing for all $\zeta \le \rho$, it is sufficient to show
$$\log\left(\frac{1-\rho}{e\lambda^2}\right) + \frac{1}{2} > \frac{\lambda^2}{4\rho}.$$
Rearranging this yields
$$\frac{\lambda^2}{4\rho} \exp\left(\frac{\lambda^2}{4\rho}\right) < \frac{1-\rho}{4\rho\sqrt{e}},$$
or equivalently $\frac{\lambda^2}{4\rho} < \mathcal{W}\left(\frac{1-\rho}{4\rho\sqrt{e}}\right)$ where $\mathcal{W}(y)$ (for $y \ge 0$) denotes the (unique) solution to $x e^x = y$. For $y > e$ we have $\mathcal{W}(y) < \log y$ and so for sufficiently small $\rho$, we are done provided that $\lambda < 2 \sqrt{-\rho \log \rho - O(\rho)}$ as desired.

It remains to consider the other case: the $\zeta$ values for which (\ref{eq:zeta-cond}) fails. In this case the infimum over $t$ in (\ref{eq:t-opt}) is optimized at $t^2 = \frac{1}{\rho} \sqrt{2 \zeta G(\zeta)}$ and (\ref{eq:t-opt}) becomes
$$\frac{\lambda^2}{2} < \left(\sqrt{G(\zeta)} + \frac{\rho}{\sqrt{2\zeta}}\right)^2$$
and so it is sufficient to show $\frac{\rho^2}{2\zeta} > -2\rho \log \rho$. Assume on the contrary that $\zeta \ge \frac{\rho}{4\log(1/\rho)}$. Combining the bound (\ref{eq:G-bound}) with the fact that (\ref{eq:zeta-cond}) fails yields
$$\frac{\zeta^2}{\rho^2} > 2 \log\left(\frac{\zeta(1-\rho)}{e\rho^2}\right).$$
Plugging in $\zeta \ge \frac{\rho}{4\log(1/\rho)}$ gives
$$1 \ge \frac{\zeta^2}{\rho^2} > 2 \log\left(\frac{\zeta(1-\rho)}{e\rho^2}\right)
\ge 2 \log\left(\frac{(1-\rho)}{4e\rho \log(1/\rho)}\right) = 2 \log(1/\rho) - 2 \log\log(1/\rho) + O(1) = \Theta(\log(1/\rho))$$
as $\rho \to 0$, which is a contradiction for sufficiently small $\rho$. This completes the proof.
\end{proof}

\section{Proof of MAP upper bound}\label{app:proof-entropy-upper-bound}

In this appendix we prove the following upper bound:
\begin{repproposition}{prop:entropy-upper-bound}
Fix $d\ge 2$ and let $\cX$ be a prior on the unit sphere in $\RR^n$, with bounded R\'enyi entropy density: there exists $\delta > 0$ with 
$$ \limsup_{n \to \infty} \frac1n H_{1-\delta}(\cX) < \infty. $$
Suppose furthermore that the varentropy $V(\cX) = \Var_{x \sim \cX}[-\log \Pr_\cX(x)]$ is $o(n^2)$.
Let $s = \limsup_{n \to \infty} \frac{1}{n} H_1(\cX)$, the Shannon entropy density. Then when $\lambda > 2\sqrt{s}$, there exists a hypothesis test distinguishing $\WT(d,\lambda,\cX)$ from $\WT(d)$ with probability $o(1)$ of error (i.e.\ strong detection is possible) and furthermore there exists an estimator that achieves nontrivial correlation with the spike (i.e.\ weak recovery is possible).
\end{repproposition}
\begin{proof}
Given a tensor $T$, consider the statistic
\begin{equation}
\label{eq:map}
m = \sup_{v \in \supp \cX}\; \langle T, v^{\otimes d} \rangle + \frac{2}{n\lambda} \log \Pr_\cX(v).
\end{equation}
We will show that thresholding $m$ suitably yields a hypothesis test that distinguishes the spiked and unspiked models with $o(1)$ probability of error of either type.

Suppose first that $T$ is drawn from the spiked model $\WT(d,\lambda,\cX)$ with spike $x$. Then
$$ \EE m \geq \EE\left[ \langle T, x^{\otimes d} \rangle + \frac{2}{n\lambda} \log \Pr_\cX(x) \right] = \lambda - \frac{2s}{\lambda}, $$
and moreover, $\langle T, x^{\otimes d} \rangle + \frac{2}{n\lambda} \log \Pr_\cX(x)$ concentrates to within $o(1)$ with high probability, using a Gaussian tail bound on the noise and Chebyshev on the entropy term, using the varentropy assumption. Thus for any fixed $\eps > 0$, $m \ge \lambda - 2s/\lambda - \eps$ with probability $1 - o(1)$.

On the other hand, suppose that $T$ is drawn from the unspiked model $\WT(d)$. Then
\begin{align*}
    &\problr{m \geq \lambda - \frac{2s}{\lambda} - 2\eps}
    \leq \sum_v \Pr\left[ \langle T, v^{\otimes d} \rangle \geq \lambda - 2\eps - \frac{2s}{\lambda} - \frac{2}{n\lambda} \log \Pr_\cX(v) \right] \\
    &\qquad\qquad\qquad= \sum_v \Pr\left[ \cN(0,2/n) \geq \lambda - 2\eps - \frac{2s}{\lambda} - \frac{2}{n\lambda} \log \Pr_\cX(v) \right] \\
    &\qquad\qquad\qquad\leq \sum_v \exp\left( -\frac{n}{4} \left(\lambda - 2\eps - \frac{2s}{\lambda} - \frac{2}{n\lambda} \log \Pr_\cX(v) \right)^2 \right) \\
    &\qquad\qquad\qquad= \exp\left(-\frac{n}{4} \left( (\lambda-2\eps)^2 - 4s\, \frac{\lambda-2\eps}{\lambda} \right) \right) \sum_v \exp\left( \frac{\lambda-2\eps}{\lambda} \log\Pr_\cX(v) - \frac{1}{n \lambda^2} (ns - \log \Pr_\cX(v))^2 \right) \\
    &\qquad\qquad\qquad\leq \exp\left(-\frac{n}{4} \left( (\lambda-2\eps)^2 - 4s\, \frac{\lambda-2\eps}{\lambda} \right) \right) \sum_v \Pr_\cX(v)^{1-2\eps/\lambda} \\
    &\qquad\qquad\qquad= \frac{2\eps}{\lambda} \exp\left(-\frac{n}{4} \left( (\lambda-2\eps)^2 - 4s\, \frac{\lambda-2\eps}{\lambda} - \frac{2\eps}{n\lambda} H_{1-2\eps/\lambda}(\cX) \right) \right).
\end{align*}
As $H_\alpha$ is monotonically decreasing in $\alpha$, we have that $\frac1n H_{1-2\eps/\lambda}(\cX) \leq \frac1n H_{1-\delta}(\cX)$ is bounded for sufficiently small $\eps$ as $n \to \infty$; thus as $\eps \to 0$, we have that
$$ -\frac{2\eps}{n\lambda} H_{1-2\eps/\lambda}(\cX) \to 0 $$
uniformly in $n$. Thus, so long as $\lambda > 2 \sqrt{s}$, we can choose $\eps > 0$ such that this probability of error is $o(1)$. Hence, thresholding the statistic $m$ at $\lambda - \frac{2s}{\lambda} - \eps$, we obtain a hypothesis test with $o(1)$ probability of error.

In order to perform weak recovery given $T = \lambda x^{\otimes d} + W$, output the $v$ that maximizes (\ref{eq:map}). From the analysis above we have with probability $1-o(1)$, $\langle T,v^{\otimes d} \rangle + \frac{2}{n\lambda} \log \Pr(v) \ge \lambda - \frac{2s}{\lambda}- \eps$ and $\langle W,v^{\otimes d} \rangle + \frac{2}{n\lambda} \log \Pr(v) \le \lambda - \frac{2s}{\lambda} - 2\eps$. This means we have
$$\lambda - \frac{2s}{\lambda} - \eps \le \langle T, v^{\otimes d}\rangle + \frac{2}{n\lambda}\log \Pr(v) = \lambda \langle x,v \rangle^d + \langle W,v^{\otimes d} \rangle + \frac{2}{n\lambda}\log \Pr(v) \le \lambda \langle x,v \rangle^d + \lambda - \frac{2s}{\lambda} - 2\eps$$
and so $\langle x,v \rangle^d \ge \eps/\lambda$, implying weak recovery.
\end{proof}

\end{document}